\documentclass[12pt, a4paper, reqno]{article}

\usepackage[margin=30mm]{geometry}
\usepackage{amsmath, amsthm, amsfonts, amssymb, mathrsfs, hyperref}

\newcommand{\R}{\mathbb R}

\newcommand{\be}{\begin{equation}}
\newcommand{\ee}{\end{equation}}
\newcommand{\deriv}[2]{\frac{d #1}{d #2}}
\newcommand{\pderiv}[2]{\frac{\partial #1}{\partial #2}}

\newtheorem{thm}{Theorem}[section]

\newtheorem{prop}[thm]{Proposition}

\newtheorem{rem}{Remark}

\newtheorem{exmp}[thm]{Example}

\numberwithin{equation}{section}

\usepackage{fancyhdr} 
\usepackage{sectsty}
\allsectionsfont{\sffamily\mdseries\upshape}

\usepackage[nottoc,notlof,notlot]{tocbibind} 
\usepackage[titles]{tocloft} 

\newcommand{\userName}{Ifan Johnston\thanks{
I. Johnston is supported by EPSRC as part of the MASDOC DTC at the University of Warwick. Grant No. EP/HO23364/1.} \footnote{This article was published on 26/06/2019 in the journal Fractal and Fractional}  \quad Vassili Kolokoltsov}

\begin{document}
\title{Green's function estimates for time fractional evolution equations}
\author{\userName}
\maketitle

\begin{abstract}
We look at estimates for the Green's function of time-fractional evolution equations of the form $D^{\nu}_{0+*} u = Lu$, where $D^{\nu}_{0+*}$ is a Caputo-type time-fractional derivative, depending on a L\'evy kernel $\nu$ with variable coefficients, which is comparable to $y^{-1-\beta}$ for $\beta \in (0, 1)$, and $L$ is an operator acting on the spatial variable. First, we obtain global two-sided estimates for the Green's function of $D^{\beta}_0 u = Lu$ in the case that $L$ is a second order elliptic operator in divergence form. Secondly, we obtain global upper bounds for the Green's function of $D^{\beta}_0 u=\Psi(-i\nabla)u$ where $\Psi$ is a pseudo-differential operator with constant coefficients that is homogeneous of order $\alpha$. Thirdly, we obtain local two-sided estimates for the Green's function of $D^{\beta}_0 u = Lu$ where $L$ is a more general non-degenerate second order elliptic operator. Finally we look at the case of stable-like operator, extending the second result from a constant coefficient to variable coefficients. In each case, we also estimate the spatial derivatives of the Green's functions. To obtain these bounds we use a particular form of the Mittag-Leffler functions, which allow us to use directly known estimates for the Green's functions associated with $L$ and $\Psi$, as well as estimates for stable densities. These estimates then allow us to estimate the solutions to a wide class of problems of the form $D^{(\nu, t)}_0 u = Lu$, where $D^{(\nu, t)}$ is a Caputo-type operator with variable coefficients.
\end{abstract}
\newpage
\tableofcontents
\section{Introduction}
The area of fractional evolution equations has become hugely popular in recent decades, due to its ability to better model real-world phenomena compared to their non-fractional counterpart which usually model local behaviour. The nature of fractional in-time operators (respectively in space) allow us to model, for example, processes that exhibit some kind of memory (resp. non-local interactions). The processes associated with time-fractional evolution models (in the case of a local spatial operator, they are time-changed diffusion processes), possess some remarkable properties. For some motivation, let us focus for the moment on the particular case of
\[D^\beta_* u = \Delta u,\]
where $D^\beta_*$ is the Captuo fractional derivative in time, $\beta \in (0, 1)$ and $\Delta$ is the Laplacian operator (a second order uniformly elliptic operator). This time-fractional diffusion equation is widely used to model anomalous diffusions which exhibit subdiffusive behaviour, which is due to the diffusive particles being trapped. 
Such fractional time diffusion equations also arise as a scaling limit of random conductance models (random walks in random environments). This point of view is particularly interesting, since the limiting process is a non-Markovian process which arises as the scaling limit of Markovian process, see \cite{barlow2011convergence,meerschaert2012stochastic,leonenko2013correlation} and references therein for a probabilistic account of models related to fractional calculus. Recently in \cite{hairer2018fractional}, the authors discussed how a fractional kinetic process (with $\beta = \frac{1}2$) emerges as the intermediate time behaviour of perturbed cellular flows. For an extensive account of physical applications see \cite{richard2014fractional}, \cite{mainardi2010fractional} or \cite{tarasov2011fractional}. For early applications of continuous time random walks and fractional calulus arising in finance, see the series of articles \cite{scalas2000fractional, mainardi2000fractional, gorenflo2001fractional}.

Recently much attention has been given to the Green's function of fractional differential equations. In \cite{chen2018heat}, the authors obtain two-sided estimates for the Green's functions of fractional evolution equations, under the assumption that the Green's function of the spatial operator satisfies global (in time) two-sided estimates.

In \cite{grigor2008dichotomy}, the authors explore the general structure of two-sided estimates for the transition probabilities associated with local or non-local Dirichlet forms. They show that the bounds for transition probabilities associated with local
 Dirichlet forms will always be of exponential type, and for those associated with non-local Dirichlet forms the bounds will be of polynomial type. Even more recently in~\cite{deng2018exact} they give some exact asymptotic formulas for the Green's function of fractional evolution equations.
The authors in \cite{kelbert2016weak} study error estimates for continuous time random walk (CTRW) approximation of classical fractional evolution equations, for which the heat kernel estimates for $D^\beta u = \Delta u$ and $D^\beta u = \psi(-i\nabla)u$, where $\psi(-i\nabla)$ generates a symmetric stable process, are obtained as a by-product. 

In \cite{eidelman2004cauchy} the authors use the parametrix method (or Levi method) to study the equation $D^\beta u(t,x) - Bu(t, x)=f(t,x)$, where the operator $B$ is a uniformly elliptic second order differential operator (which we look at in Theorem \ref{thm:fracdiffusionlocal}) with bounded continous real-valued coefficients. This is done by looking first at the constant coefficient case then using these estimates to study the variable coefficient case. In the articles \cite{kochubei2019random, kochubei2018random}, the authors study the long-time behaviour of the Cesaro mean of the heat kernel of subordinated processes and for this they use a version of a Karamata-Tauberian theorem. The long-time behaviour of solutions to space-time fractional diffusion equations are also considered in \cite{cheng2017asymptotic}. 
In the articles \cite{kemppainen2017representation, cheng2017asymptotic, kim2016asymptotic}, the authors consider the space-time fractional diffusion equation, which involves a Caputo fractional derivative in time and a fractional Laplacian as the spatial operator. This falls under the case of our Theorem \ref{thm:stableglobal}, where we consider as the spatial operator a (non-isotropic) pseudo-differential operator with symbol
\[\psi_\alpha(\xi) = -|\xi|^\alpha w(\xi/|\xi|), \quad \xi\in \R^d,~\alpha \in (0, 2),~w \in C^k(\mathbb S^{d-1}),\]
see (\ref{eq:defsymbolgeneral}). Since this case involves constant coefficients, we could have used the machinery of the Fox H function like they do in \cite{kemppainen2017representation, kim2016asymptotic} (among many others that use them) or Laplace transform arguments. However these approaches would not work for our other results, since the operators involved have variable coefficients. 

Diffusion processes in random environments are also closely related objects, and in fact there are many works looking at estimates for the heat kernels of such processes; for example in \cite{cabezas2015sub}, the authors obtain sub-gaussian bounds for the transition kernel of a random walk in a random environment.

This article is structured as follows: in Section 2 we begin by recalling definitions that will be used throughout the article. The topic of generalised fractional calculus is briefly covered, which is motived by a probabilistic generalisation of Caputo fractional derivatives. We also recall some important estimates, namely the Aronson estimates for the fundamental solutions of second order parabolic equations and asymptotic estimates for stable densities. In Section 3 (resp. Section 4) we obtain global estimates (resp. local estimates) for the Green's function of fractional evolution equations of the form
\be\label{eq:generalfracevol}
\left\{
\begin{array}{rll}
D^\beta_* u(t, x) = &L u(t,x),& \text{in } (0, \infty)\times \R^d\\
u(t,x) =& Y(x),& \text{in } \{0\}\times \R^d,
\end{array}
\right.
\ee
where $L$ is some differential operator acting on the spatial variable $x$.
In the final section we discuss solutions to generalised evolution equations of the form
\[-D^{(\nu, t)} u(t, x) = Lu(t, x),\]
where $D^{(\nu,t)}$, acting on the time variable, is the generator of an increasing process, which is comparable to a $\beta$-stable subordinator, and $L$ is a generator of a strongly continuous contraction semigroup. We then conclude by summarising the main results and commenting on some applications.

We summarise briefly some of the results obtained for various spatial operators.
\begin{itemize}
\item Theorem \ref{thm:globaldivergenceestimate}: When the spatial operator is given by a second order uniformly elliptic operator in \emph{divergence form},
\[Lu = \nabla\cdot(A(x) \nabla u),\]
we obtain the following two-sided estimates for the Green's function $G^{(\beta)}(t,x,y)$ of (\ref{eq:generalfracevol}). For $d \geq 3$ and $\beta\in (0, 1)$ we find:
\be
\label{intro:globalfracdiff}
G^{(\beta)}(t,x,y) \asymp C\min\left(t^{-\frac{d\beta}2}\Omega^{1-\frac d2}, t^{-\frac{d\beta}2}\Omega^{-\frac d2 \left(\frac{1-\beta}{2-\beta}\right)}\exp\{-C_\beta \Omega^{\frac{1}{2-\beta}}\}\right),
\ee
for $(t, x, y) \in (0, \infty) \times \R^d\times \R^d$ where $\Omega := |x-y|^2 t^{-\beta}$.

\item Theorem \ref{thm:stableglobal}: When the spatial operator is a non-isotropic pseudo-differential operator  $\Psi_\alpha$, with homogeneous symbol of the form (cf. Equation (\ref{eqdef:wspectralmeasure}))
\[\psi_\alpha(\xi) = |\xi|^\alpha w_\mu(\xi/|\xi|), \quad\xi \in \R^d,~\alpha\in(0,2),~w_\mu\in C^k (\mathbb S^{d-1}),\]
we obtain two-sided estimates for the Green's function $G^{(\beta)}_{\psi_\alpha}$ of (\ref{eq:generalfracevol}). For $d > \alpha > 0$ and $\beta\in(0, 1)$:
\be
\label{intro:globalfracstable}
G^{(\beta)}_{\psi_\alpha}(t,x-y) \asymp C \min\left(t^{-\frac{d\beta}\alpha} \Omega^{1-\frac d\alpha}, t^{-\frac{d\beta}\alpha} \Omega^{-1-\frac d\alpha}\right)
\ee
for $(t,x-y) \in (0, \infty) \times\R^d$, where $\Omega := |x-y|^\alpha t^{-\beta}$.
\end{itemize}
It is well known that the Green's functions of the above operators $L$ and $\Psi_\alpha$ satisfy two-sided estimates for all $(t, x, y)\in (0, \infty)\times\R^d\times \R^d$, see \cite{aronson1967bounds} or \cite{Stroock1988} for the first case, \cite{kolokoltsov2019differential} or \cite{eidelman2004analytic} for the second. Note that in the case of a lower bound for the Green's function of $\Psi_\alpha$, one also needs to assume that its spectral density $\mu$ is strictly positive, and restrict $\alpha \in (0, 2)$. 

In Section \ref{sec:localestimates} we consider more general spatial operators, whose Green's functions satisfy only local (in time) estimates. We obtain local estimates for the Green's function of (\ref{eq:generalfracevol}) in the following cases:
\begin{itemize}
\item Theorem \ref{thm:fracdiffusionlocal}: A general second order elliptic operator $L$ with variable coefficients of the form
\[
Lu = \sum_{i, j = 1}^d a_{ij}(x)\partial_{x_i}\partial_{x_j}u + \sum_{i=1}^d b_i(x)\partial_{x_j}u + c(x) u.
\]
Unsurprisingly we find that the estimates are the same as (\ref{intro:globalfracdiff}) but only for $(t, x, y) \in (0, T)\times \R^d\times \R^d$ for some fixed $T > 0$.
\item Theorem \ref{thm:fracevollocal}: A non-isotropic pseudo-differential operator $\psi_{\alpha, x}$ with variable coefficients, with homogeneous symbol of the form
\[\psi_{\alpha, x}(x, \xi) = |\xi|^\alpha w_\mu(x, \xi/|\xi|),\]
where $w_\mu$ is for each fixed $x$ a continuous function on the surface of the sphere $\mathbb S^{d-1}$. Note again that the spectral density associated with $\psi_\alpha$ (for fixed $x$) must be strictly positive, in order to use two-sided estimates for the Green's function $G_{\psi_\alpha, x}$.
Again we obtain the same estimates as (\ref{intro:globalfracstable}) but for $(t, x, y) \in (0, T)\times \R^d \times \R^d$ for some fixed $T> 0$.
\end{itemize}
For each of the four cases above, we also obtain estimates for the spatial derivatives of the Green's functions of the fractional evolution equations.

Finally, in the last section we turn our attention to generalised fractional evolutions. This is when one replaces the standard fractional time derivative with a weighted mixture of fractional derivatives,
 \[D^{(\nu)}_{0+*} f(t) = -\int_0^t(f(t-s) - f(t))\nu(t, \mathrm ds) - \int_t^\infty (f(0) - f(t))\nu(t, \mathrm ds).\]
The main result is that the solution to such generalised fractional evolutions can be estimated (by means of our Green's function estimates) by the solutions to classical fractional evolution equations.

\section{Preliminaries}
\vspace{-6pt}
\subsection{Estimates and Stable Processes}

Throughout the article, we will use the notation $f(x) \asymp g(x)$ in $D$, which means that there exists constants $C, c > 0$ such that $f$ satisfies the following two-sided estimate, 
\[cg(x) \leq f(x) \leq Cg(x), \quad \forall x \in D,\]
for some region $D$. The notation $f(x) \sim g(x)$ for $x\rightarrow \infty$ means that 
\[\frac{f(x)}{g(x)} \rightarrow 1, \quad \text{ as } x\rightarrow \infty.\]
Then for each $M > 0$ there exists a constant $C > 0$ such that
\[cg(x) \leq f(x) \leq Cg(x), \quad x \in (M, \infty).\]

Similarly, the notation $f(x) \sim g(x)$ for $x\rightarrow 0$ means 
\[\frac{f(x)}{g(x)} \rightarrow 1,\quad \text{ as } x\rightarrow 0.\]
Then for each $m > 0$ there exists a $c > 0$ such that
\[cg(x) \leq f(x) \leq c g(x), \quad x \in (0, m).\]

If both $f$ and $g$ on $\R_+$ are positive, bounded and satisfy $f(x) \sim g(x)$ for $x \rightarrow \infty$ (resp. $x\rightarrow 0$), then $f(x) \asymp g(x)$ in $(M, \infty)$ for any $M > 0$ (resp. in $(0, m)$ for any $m  <\infty$). See the excellent De Bruijn book \cite{de1970asymptotic} for more details on asymptotic analysis.

The fundamental solutions $Z(t, x; \tau, \xi)$ of the Cauchy problem for the uniformly parabolic~equations
\[\partial_t u - \{a_{ij}(t, x) \partial_{x_i}\partial_{x_j} u + b_i(t, x) \partial_{x_i}u + c(t, x) u\} = 0, \quad u(0, x) = \delta(x-\xi)\]
with bounded and uniformly H\"older continuous coefficients in $x$ defined on $(0, T] \times \R^d$ are known,~\cite{porper1984two}, to satisfy the two-sided estimate
\[Z(t, x;\tau, \xi) \asymp (t- \tau)^{-d/2} \exp\left\{-c\frac{(x-\xi)^2}{t-\tau}\right\}.\] 

We will assume that the coefficients do not depend on time, so that the fundamental solution is just a function of $t, x$ and $y$.
On the other hand, Aronson \cite{aronson1967bounds} obtained global two-sided estimates for the fundamental solution  $G(t, x, y)$ of the divergence equation
\be\label{eq:paradivform}
\partial_t u = \partial_{x_i}(a_{ij}(x) \partial_{x_j} u).
\ee

Assume that the coefficients satisfy the uniform ellipticity condition: there exists $\mu \geq 1$ such that
\be\label{eq:ellipticity}\mu^{-1} |\xi|^2 \leq a_{ij}(x) \xi_i\xi_j \leq \mu |\xi|^2, \quad \text{ for all } \xi\in \R^d.\ee

Further assuming that the coefficients in (\ref{eq:ellipticity}) are continuous, then there exists constants $C_1, C_2, c_1$ and $c_2$ such that for $(t, x, y) \in (0, \infty)\times \R^d\times \R^d$,
\begin{equation}\label{eq:nash-aronson}
c_1 t^{-d/2}\exp\left\{-c_2\frac{|x-y|^2}{t}\right\} \leq G(t, x, y) \leq C_1 t^{-d/2}\exp\left\{-C_2\frac{|x-y|^2}{t}\right\}.
\end{equation}

For a discussion on divergence and non-divergence equations, see for example \cite{evans2010partial}.

Note that throughout $C$, $\tilde{C}$, $c$ or $\tilde{c}$ denotes some constants. If we wish to stress what they depend on, say $\alpha, \beta$ or $T$, we write $C_{\alpha, \beta, T}$ for example.

Let us recall some basic facts about stable densities; our standard references for these are \cite{zolotarev1986one,kolokoltsov2011markov}. The characteristic function of the general (up to a shift) one-dimensional stable law with index of stability $\beta\in (0, 2)$ (but $\beta \neq 1$) is given by
\[\phi_\beta(y) =\exp\{-\sigma |y|^\beta e^{i\frac{\pi}{2}\gamma \text{ sgn }\gamma}\}, \quad y \in \R,\]
where the parameter $\gamma\in [-1, 1]$ measures the skewness of the distribution and $\sigma > 0$ is the scale parameter. The probability density corresponding to the characteristic function $\phi_\beta$, which we denote by $w_\beta(x;\gamma, \sigma)$, is given by the following Fourier transform:
\[w_\beta(x;\gamma, \sigma) = \frac{1}{2\pi}\int_{\R} \exp\{-ixy - \sigma | y|^\beta e^{i\frac{\pi}{2} \gamma \text{ sgn }\gamma}\}~\mathrm dy.\]

We will be using totally positively skewed ($\gamma = 1$) normalised ($\sigma = 1$) stable densities, which we denote by $w_\beta(x)$ and they are given by
\[w_\beta(x) = \frac{1}{\pi}\Re \int_0^\infty \exp \{-ixy - |y|^\beta\}~\mathrm dy,\]
where $\Re(z)$ is the real part of $z\in \mathbb C$.
We will be using the asymptotic behaviour (as $x\rightarrow 0$ and $x \rightarrow \infty$) of the stable densities $w_\beta(x)$, so we state them now, \cite{uchaikin2011chance} (Theorem 5.4.1).
\begin{prop}
The stable densities $w_\beta(x)$ have the following asymptotic behaviour
\begin{equation}\label{prop:asymptoticsstable}
w_\beta(x) \sim \tilde{c}_\beta\left\{
\begin{array}{lc}
x^{-1-\beta},& \text{ as } x\rightarrow \infty,\\
f_{\beta}(x) := x^{-\frac{2- \beta}{2(1-\beta)}}\exp\left\{-c_\beta x^{-\frac{\beta}{1-\beta}}\right\}, & \text{ as } x \rightarrow 0,
\end{array}
\right.
\end{equation}
where $c_\beta = (1-\beta)\beta^{\frac{\beta}{1-\beta}}$.
\end{prop}

A key point to note from the above asymptotic behaviour, is that since $w_\beta$ is bounded and strictly positive for $x > 0$, this gives us a two-sided estimate for $w_\beta(x)$ for all $x\in(0, \infty)$. More precisely for $\beta \in (0, 1/2)$ we have
\be\label{eq:stabledensityineq}w_\beta (x) \asymp c_\beta \min\left(x^{-1-\beta} ,x^{-\frac{2- \beta}{2(1-\beta)}}\exp\left\{-c_\beta x^{-\frac{\beta}{1-\beta}}\right\}\right),\ee
and for $\beta\in[1/2,1)$
\[
w_\beta(x) \asymp \tilde{c}_\beta\left\{
\begin{array}{lc}
x^{-1-\beta},& \text{ for } x \in (1, \infty),\\
x^{-\frac{2- \beta}{2(1-\beta)}}\exp\left\{-c_\beta x^{-\frac{\beta}{1-\beta}}\right\}, & \text{ for } x \in (0,1).
\end{array}
\right.
\]

For $\alpha \in (0, 2)$, the general symmetric stable density in $\R^d$ (up to a shift) has the form
\be\label{eqdef:spectralmeasure}
\phi_\alpha(p) = \exp\left\{-|p|^\alpha \int_{\mathbb S^{d-1}}|(p/|p|, s)|^\alpha \mu(\mathrm ds)\right\},
\ee
where the (finite) measure $\mu$ on $\mathbb S^{d-1}$ is called the spectral measure, \cite{kolokoltsov2000symmetric}. Let $w$ be a function on $\mathbb S^{d-1}$ given by
\be\label{eqdef:wspectralmeasure}
w_\mu(p) = \int_{\mathbb S^{d-1}}|(p, s)|^\alpha \mu(\mathrm ds),
\ee
so that
\be\label{eq:defsymbolgeneral}
\psi_\alpha(p) := \log \phi_\alpha(p) = -|p|^\alpha w_\mu(p/|p|), \quad p\in \R^d.
\ee

Note that $\psi_\alpha$ is the symbol of a pseudo-differntial operator $\Psi_\alpha(-i\nabla)$ which we will study later. When $\mu$ is the uniform measure on $\mathbb S^{d-1}$ we have $w(p) \equiv 1$, and the operator $\Psi_\alpha(-i\nabla)$ is just the fractional Laplacian $(-\Delta)^\alpha$ with symbol $\psi_\alpha(\xi) = -|\xi|^\alpha$.

\subsection{Fractional Derivatives and Their Extensions}
Fractional derivatives at $a\in \R$ of order $\beta \in (0,1)$ can be viewed probabilistically as the generators of a $\beta$-stable process, interrupted on an attempt to cross a boundary point $a$. This interpretation was extensively explored and extended in \cite{kolokoltsov2015fully}, and we recall some of the details here. For a broader background in classical fractional calculus, see any text on fractional calculus \cite{samko1993fractional, diethelm2010analysis, kilbas2006theory, podlubny1998fractional}, for example. 

The fractional Riemann-Liouville (RL) integral of order $\beta\in (0, 1)$ is given by
\[I^\beta_a f(x) = \frac{1}{\Gamma(\beta)}\int_a^x(x-t)^{\beta-1}f(t)\mathrm dt,\quad  x > a,\]
and the Caputo fractional derivative of order $\beta \in (0, 1)$ is given by
\[D^\beta_{a+*}f(x) = \frac 1{\Gamma(-\beta)}\int_0^{x-a} \frac{f(x-z) - f(x)}{z^{1+\beta}}\mathrm dz + \frac{f(x) - f(a)}{\Gamma(1-\beta)(x-a)^\beta}.\]

The fractional derivative in \emph{generator from}, is written as
\[D^\beta_{+} f(x) = \frac 1{\Gamma(-\beta)}\int_0^\infty \frac{f(x-z) - f(x)}{z^{1+\beta}}\mathrm dz,\]
which equals the Caputo derivative for $a=-\infty$. 
A possible extension for these fractional derivative operators which is widely used in the literature, are various mixtures of these derivatives, for example
\[\sum_{i=1}^N a_i \deriv{^{\beta_i}f}{x^{\beta_i}}\quad \text{ or } \quad \int_0^1 \deriv{^\beta f}{x^\beta}\mu(\mathrm d\beta).\]

In this article we will look at weighted mixed fractional derivatives, given by
\[D^{(\nu)}_+ f(t) = -\int_0^\infty (f(t-s) - f(t))\nu(t,\mathrm ds),\]
with some positive kernel $\nu(t, \cdot)$ on $\{s: t > 0\}$ satisfying the one-sided L\'evy-condition
\[\sup_t\int_0^\infty \min(1, s) \nu(t,\mathrm ds) < \infty.\]

The $-$ sign in the definition of $D^{(\nu)}_+$ is to comply with the standard notation for fractional derivatives, so that $D^{\beta}_+ = D^{(\nu)}_+$ with $\nu(t,y) = -1/\left[\Gamma(-\beta) y^{1+\beta}\right]$.

Notice that the Caputo derivative $D^{\beta}_{a+*}$ is obtained from $D^\beta_+$ by the restriction of its action on the space $C^1([a, \infty))$ considered as the subpsace of $C(\R)$ by extending their values as constants to the left of $a$. Then looking for a generalised Captuo derivative arising from $D^{(\nu)}_+$ we define
\[D^{(\nu)}_{a+*}f(t) = -\int_0^{t-a}(f(t-s) - f(s))\nu(t, \mathrm ds) - \int_{t-a}^\infty (f(a) - f(s)) \nu(t, \mathrm ds).\]

These extensions were initially suggested by the second author in \cite{kolokoltsov2015fully}, on the basis of the following probabilistic interpretation. It is well known that the fractional derivatives $-d^\beta/dx^\beta$ for $\beta\in(0,1)$  generate of stable L\'evy subordinators with the inverted direction (i.e., decreasing processes). Then the Caputo derivatives $D^\beta_{a+*}$ describe the modifications of the stable subordinators obtained by forbidding them to cross the boundary $x=a$ for some $a\in \R$. Applying this modification to the generalised operators $D^{(\nu)}_+$ one obtains $D^{(\nu)}_{a+*}$ which are then the generators of a Markov process interrupted on an attempt to cross the boundary point $a$.

The corresponding generalised fractional integrals arising from the generalised fractional derivative $D^{(\nu)}_+$ can be defined in a few different ways depending on which point of view one chooses: probability, semigroup theory or generalised functions/$\Psi$DE theory. The objects end up being the same, but we will stick to the semigroup theory prespective in this article, see the upcoming review \cite{kolokoltsov2019mixed} for a full account. %
In terms of operator semigroups, the operator $I^\beta_{-\infty}$ is the potential operator of the semigroup generated by $-D^{\beta}_+$. In other words, it is the limit of the resolvent operator
\[R_\lambda = (\lambda - D^\beta_+)^{-1},\]
as $\lambda \rightarrow 0$. Then since $I^\beta_a$ is the reduction of $I^\beta_{-\infty}$  to the space $C_{kill(a)}([a, \infty))$ of functions vanishing to the left of $a$ for any $a\in \R$, we define the \emph{generalised fractional integral} $I^\nu_a$ as the potential operator of the semigroup generated by $-D^{(\nu)}_+$  reduced to the space $C_{kill(a)}([a, \infty))$. For a background on potential operators and measures, see \cite{schilling2012bernstein} or \cite{van2012potential} (or from a probabilistic point of view, \cite{feller2008introduction}).

Let us denote by $(T^{\nu}_t)_{t\geq 0}$ the semigroup generated by the operator $-D^{(\nu)}_+$. Then for $f\in C_{kill(a)}([a, \infty))\cap C_\infty(\R)$, the potential operator $U^{(\nu)}$ of the semigroup $T^{\nu}_t$ is given by
\begin{align*}
U^{(\nu)} f(t) &= \int_0^\infty P^{\nu}_r f(t)\mathrm dr\\
&= \int_0^{\infty} \int_0^\infty f(t- s) G_{(\nu)}(r,t, \mathrm ds) \mathrm dr\\
&= \int_0^{t-a} f(t-s) \left(\int_0^\infty G_{(\nu)}(r,t,\mathrm ds)\right)\mathrm dr
\end{align*}
where $G_{(\nu)}(r,t, \mathrm ds)$ are the transition probabilities of the process generated by $D^{(\nu)}_+$. 
The potential measure is defined as the integral kernel of the potential operator, and by an abuse of notation, we denote this measure by $U^{(\nu)}(t, \mathrm ds)$. Thus the generalised fractional integral $I^{(\nu)}_a$ is given by
\[I^{(\nu)}_a f(t) = \int_0^{t-a} f(t-s)U^{(\nu)}(t,\mathrm ds)\]
where the potential measure $U^{(\nu)}(t, \mathrm ds)$ is equal to the vague limit
\[U^{(\nu)}(t,M) = \int_0^\infty G_{(\nu)}(r, t, M) \mathrm dr,\]
of the measures $\int_0^K G_{(\nu)}(r,t, \cdot)\mathrm dr, ~K \rightarrow \infty$ (see \cite{schilling2012bernstein} (p. 63)). Furthermore the $\lambda$-potential measure is defined by
\[U^{(\nu)}_{\lambda}(t, M) = \int_0^\infty e^{-\lambda r}G_{(\nu)}(r,t, M)\mathrm dr,\]
so that if $\lambda > 0$ and $g\in C_{kill(a)}([a, \infty))\cap C_\infty(\R)$, the convolution $(U^{(\nu)}_\lambda \star g)(t)$, which is given by
\be\label{eq:resolvingoperator}
(U^{(\nu)}_\lambda \star g)(t) = \int_0^{t-a} g(t-s) \int_0^\infty e^{-\lambda r} G_{(\nu)}(r,t, \mathrm ds)\mathrm dr,
\ee
is the resolving operator of the semigroup generated by $-D^{(\nu)}_+$. That is, $f(x) = (U^{(\nu)}_\lambda \star g)(x)$ is the classic solution to the equation
\[D^{(\nu)}_+ f = D^{(\nu)}_{a+*}f = -\lambda f + g.\]

This also holds for $\lambda = 0$, and so the potential operator with kernel $U^{(\nu)}(t, \mathrm dy)$,
\[(U^{(\nu)}\star g)(x) = I^{(\nu)}_a g(x),\]
represents the classical solution to the equation
\[D^{(\nu)}_{a+*}f = g,\]
on $C_{kill(a)}([a, b])$.
\begin{exmp}
For the case $\nu(t, \mathrm dy) = -1/[\Gamma(-\beta)y^{1+\beta}]\mathrm dy$, (\ref{eq:resolvingoperator}) says that
\[f(t) = \int_0^{t-a} g(t-s)\int_0^\infty e^{-\lambda r} G_{\beta}(r, s)\mathrm ds \mathrm dr,\]
where $G_\beta(r, s)$ are the transition densities of a $\beta$-stable subordinator, is the solution to the linear fractional equation
\[D^{\beta}_+ f(x)= -\lambda f(x)+ g(x), \quad f(a) = f_a.\]

On the other hand, it is well known that the solution to such linear fractional equations are given by
\begin{align*}
f(x) &= E_\beta(-\lambda(x-a)^\beta)f_a + \beta\int_a^x g(z)(x-z)^{\beta-1} E^{'}_\beta(-\lambda(x-z)^\beta)\mathrm dz\\
& = E_\beta(-\lambda(x-a)^\beta)f_a + \beta\int_0^{x-a} g(x-y)y^{\beta-1} E^{'}_\beta(-\lambda y^\beta)\mathrm dy,
\end{align*}
where $E_\beta(z)$ is the Mittag-Leffler function
\be\label{eqdef:MLSeries}E_\beta(z) = \sum_{k=0}^\infty \frac{z^k}{\Gamma(k\beta + 1)}, \quad z \in \mathbb C.\ee

Thus we have
\[U^\beta_\lambda (t) = \int_0^\infty e^{-\lambda r} G_\beta(r, t)\mathrm dr = \beta t^{\beta - 1} E^{'}_\beta(-\lambda t^\beta),\]
which is equivalent to the Zolotarev-Pollard formula for the Mittag-Leffler function in terms of the transition densities of stable subordinators, which is of great importance to this article,
\be\label{eq:mittagleffLT}
E_\beta (s) = \frac 1\beta \int_0^\infty e^{sx}  x^{-1-1/\beta} G_\beta(1, x^{-1/\beta})~\mathrm dx.
\ee
\end{exmp}

For the rest of this article we will take $a=0$ and write $D^\beta_0 := D^\beta_{0+*}$ and $D^{(\nu)}_0 := D^{(\nu)}_{0+*}$ to simplify formulas, but keep in mind that the boundary point $0$ can be replaced with $a\in \R$ by a shift in the time~coordinate.

As noted in \cite{kolokoltsov2017chronological}, we can extrapolate from the case
\[D^{\beta}_0 u(t) = -\lambda u(t), \quad \lambda \in\R, \quad u(0) = u_0,\]
to the Banach-valued version
\[D^{\beta}_0 u(t,x) = L u(t,x), \quad u(0, x) = Y(x),\]
where $L$ is some operator generating a Feller semigroup. One can expect that the solution to this equation can be written in terms of an operator-valued Mittag-Leffler function,
\be
	\label{eq:solutiondivergence}
	u(t, x) = E_\beta \left(Lt^\beta\right)Y(x),
\ee
where $E_\beta(s)$ are Mittag-Leffler functions defined by (\ref{eqdef:MLSeries}). However, this series representation does not allow one to define $E_\beta(L)$ for an unbounded operator $L$. In both \cite{kolokoltsovwell2014} and \cite{kolokoltsov2017chronological} the authors find that the most convenient way to overcome this difficulty is to use the formula (\ref{eq:mittagleffLT}) for the Mittag-Leffler function. This connection between Mittag-Leffler functions, Laplace transforms and stable densities is due to Zolotarev, \cite{zolotarev1957mellin, zolotarev1961analytic, zolotarev1986one}---although preliminary versions of this formula were also noted almost a decade earlier by Pollard, \cite{pollard1948completely}. Thus formula (\ref{eq:mittagleffLT}) could be called the Pollard-Zolotarev formula. 
Notice that if an operator $L$ generates a Feller semigroup with transition densities $G(t, x, y)$, then
\[
	e^{L t^\beta z}Y(\cdot) = \int_{\R^d} G(t^\beta z, \cdot, y) Y(y)~\mathrm dy.
\]

With the help of Fubini's theorem, the solution (\ref{eq:solutiondivergence}) can be written as
\begin{align*}
	u(t, x) &= E_\beta(t^\beta L)Y(x)\\
	& = \frac 1\beta \int_0^\infty e^{L t^\beta z}Y(x)z^{-1- \frac 1\beta}w_\beta(z^{-\frac 1\beta})~\mathrm dy\\
	&= \frac{1}{\beta}\int_0^\infty \int_{\R^d} G(t^{\beta}z, x, y) Y(y) z^{-1- \frac{1}{\beta}} w_\beta(z^{-\frac{1}{\beta}})~\mathrm dy~\mathrm dz\\
	& = \int_{\R^d}\left(\frac 1\beta\int_0^\infty G(t^\beta z, x, y) z^{-1-\frac 1\beta}w_\beta(z^{-\frac 1\beta})~dz\right)Y(y)~\mathrm dy\\
	& =: \int_{\R^d} G^{(\beta)}(t, x, y) Y(y)~\mathrm dy.
\end{align*}

Thus the main aim of this article is to obtain estimates for the Green's function given by,
\be\label{eq:Green'sfuncdivdiff}
G^{(\beta)}_L(t, x, y) =: \frac{1}\beta \int_0^\infty G_L(t^\beta z, x, y)z^{-1-\frac{1}{\beta}} w_\beta( z^{\frac{1}{\beta}})~\mathrm dz,
\ee
where $G_L(z, x, y)$ is the Green's function associated with the spatial operator $L$, i.e., the fundamental solution of
\[\partial_t u = L u.\]

\section{Global Estimates}
We first look at global in time two-sided estimates for $G^{(\beta)}_L$ in two special cases. Notice in (\ref{eq:Green'sfuncdivdiff}) that the integral over the time variable $z$ ranges from $0$ to $\infty$, and so in order to perform any estimates on the term $G_L(z, x,y)$ one can only use estimates that hold for all $z\in (0, \infty)$. We begin with two such cases, when one has global in time estimates for $G_L$. Namely, when $L$ is a second order uniformly elliptic operator in divergence form or when $L$ is a homogeneous pseudo-differential operator (with constant coefficients).

\subsection{Time-Fractional Diffusion Equation: Divergence Structure}\label{sec:aronsoncase}
In this section we consider the time-fractional diffusion equation given by
\be\label{eq:fracdiffusion}
D^\beta_0 u(t, x) = Lu(t,x) := \nabla \cdot(A(x) \nabla u(t, x)), \quad u(0, x) = Y(x),
\ee
where $D^\beta_0$ is the Caputo fractional derivative acting on the time variable, and the spatial operator is a second order elliptic operator in divergence form. For the conditions on the diffusion coefficient $A$ such that $L$ generates a Feller semigroup, see \cite{Stroock1988}.
Recall that the solution of (\ref{eq:fracdiffusion}) is given by
\[u(t, x) = E_\beta(Lt^\beta)Y(x),\]
and the associated Greens function is given by
\be\label{eqdef:GrFrDif}
G^{(\beta)}(t, x, y) = \frac 1\beta \int_0^\infty G(t^\beta z, x, y) z^{-1-\frac 1\beta} w_\beta(z^{-\frac 1\beta})~\mathrm dz,
\ee
where $G(t, x, y)$ is the Greens function associated with the second order elliptic operator in divergence form, (\ref{eq:paradivform}).
We have the following two-sided estimates for the Green's function $G^{(\beta)}$, which are global in time. In the following, we use the notation $\Omega := |x-y|^2 t^{-\beta}$.
\begin{thm}\label{thm:globaldivergenceestimate}
Assume that the function $A(x)$ is measurable, symmetric and satisfies (\ref{eq:ellipticity}) for some $\mu \geq 1$. Then there exists a constant $C$ such that for $(t, x, y) \in (0, \infty) \times \R^d \times \R^d$, the Green's function $G^{(\beta)}(t, x, y)$ for the time-fractional diffusion Equation (\ref{eq:fracdiffusion}) satisfies the following two-sided estimates,
\begin{itemize}
\item For $\Omega\leq 1$,
\be\label{eq:thm1case1}G^{(\beta)}(t, x, y) \asymp C\left\{
\begin{array}{lc}
t^{-\frac \beta 2}, & d = 1,\\
t^{-\beta}(|\log\Omega| + 1),& d=2,\\
t^{-\frac{d\beta}{2}}\Omega^{1-\frac d2}, & d \geq 3.
\end{array}
\right.
\ee

\item For $\Omega \geq 1$,
\be\label{eq:thm1case2}
G^{(\beta)}(t,x,y) \asymp
Ct^{-\frac{d\beta}2}\Omega^{-\frac d 2\left(\frac{1-\beta}{2-\beta}\right)}\exp\left\{-C_\beta\Omega^{\frac 1{2-\beta}}\right\}.
\ee
\end{itemize}
\end{thm}
\begin{proof}
Let us begin by using the estimate (\ref{prop:asymptoticsstable}) for the stable density in (\ref{eqdef:GrFrDif}),
\begin{align*}
G^{(\beta)}(t,x, y) & = \frac 1\beta \int_0^\infty G(t^\beta z, x, y) z^{-1 -\frac 1\beta}w_\beta(z^{-\frac 1\beta})~\mathrm dz\\
& \asymp c_\beta \int_0^1 G(t^\beta z, x,y)~\mathrm dz + c_\beta \int_1^\infty G(t^\beta z, x, y) z^{-1-\frac 1\beta}f_\beta(z^{-\frac 1\beta})~\mathrm dz,
\end{align*}
where $f_\beta(x) = x^{-\frac{2-\beta}{2(1-\beta)}}\exp\{-c_\beta x^{-\frac \beta{1-\beta}}\}$. Next we apply Aronsons estimates (\ref{eq:nash-aronson}) to $G(t^\beta z, x, y)$,
\[
G^{(\beta)}(t, x, y) \asymp  Ct^{-\frac{d\beta}2}\int_0^1 z^{-\frac d2} \exp\{-\Omega z^{-1}\}~\mathrm dz + Ct^{-\frac{d\beta}2}\int_1^\infty z^{-\frac d2-1-\frac 1\beta} \exp\{-\Omega z^{-1}\} f_\beta(z^{-\frac 1\beta})~\mathrm dz,
\]
where $\Omega := |x-y|^2t^{-\beta}$. Making a change of variables $z = w^{-1}$ so that $\mathrm dz=-w^{-2}\mathrm dw$,
\begin{align}\label{eqprfdef:diffI1I2}
G^{(\beta)}(t, x, y) \asymp & ~ C t^{-\frac{d\beta}2}\int_1^\infty w^{\frac d2 - 2} \exp\{-\Omega w\}~\mathrm dw\nonumber\\
& + C t^{-\frac{d\beta}2}\int_0^1 w^{\frac d2 - 1 + \frac 1\beta}\exp\{-\Omega w\} f_\beta(w^{\frac 1\beta})~\mathrm dw\\
=: &~ I_1 + I_2.\nonumber
\end{align}

We now estimate $I_1$ and $I_2$ in two different cases, depending on the behaviour of $\Omega$.

\underline{\bf Case 1: $\Omega \leq 1$.}
Making a further substitution of $V = \Omega w$ in the integral $I_1$ gives us the simpler form of
\[I_1 = C t^{-\frac{d\beta}2}\Omega^{1-\frac d2}\int_\Omega^\infty V^{\frac d2 -2} \exp\{-V\}~\mathrm dV.\]

Now if $d = 1$, then we have the asymptotic behaviour
\[I_1 = t^{-\frac \beta 2}\Omega^{\frac 12} \int_\Omega^\infty V^{-\frac 32}\exp\{-V\} ~ \mathrm dV \sim t^{-\frac \beta 2}\Omega^{\frac 12} \Omega^{-\frac 12} = C t^{-\frac \beta 2}, \quad \text{ as } \Omega \rightarrow 0,\]
and in particular for $\Omega \leq 1$ there exists a constant $C > 0$ such that
\[I_1 \asymp Ct^{-\frac \beta 2}.\]

If $d=2$, then we see logarithmic behaviour,
\be\label{eqprf:i1log}
I_1 = C t^{-\beta} \int_\Omega^\infty V^{-1} \exp\{-V\}~C \mathrm dV \sim t^{-\beta}(|\log \Omega| + 1), \quad \Omega \rightarrow 0,
\ee
and in particular for $\Omega \leq 1$ there exists a constant $C> 0$ such that
\[I_1 \asymp C t^{-\beta}(|\log \Omega| + 1).\]

If $d\geq 3$, then the integral is the so-called upper incomplete gamma function, and has the asymptotic behaviour
\[I_1= C t^{-\frac{d\beta}2}\Omega^{1-\frac d2}\int_\Omega^\infty V^{\frac d2 - 1} \exp\{-V\}~ \mathrm dV \sim Ct^{-\frac{d\beta}2}\Omega^{1-\frac d2} \Gamma\left(\frac d2 - 1\right), \quad \text{ as } \Omega \rightarrow 0,\]
and in particular for $\Omega \leq 1$ there exists a constant $C> 0$ such that the two-sided estimate
\[I_1 \asymp C t^{-\frac{d\beta}2}\Omega^{1-\frac d2},\]
holds. Thus we have the following two-sided estimate for $I_1$,
\[ I_1 \asymp C\left\{
\begin{array}{lc}
t^{-\frac \beta 2}, & d = 1,\\
t^{-\beta}(|\log\Omega| + 1),& d=2,\\
t^{-\frac{d\beta}{2}}\Omega^{1-\frac d2}, & d \geq 3.
\end{array}
\right.
\]

Turning to the integral $I_2$,
\begin{align}\label{eqprf:diffI2moment}
I_2 = & ~C t^{-\frac{d\beta}2}\int_0^1 w^{-\frac d2 - 1 + \frac 1\beta} \exp\{-\Omega w\} f_\beta(w^{\frac 1\beta})~\mathrm dw\nonumber\\
=& ~C t^{-\frac{d\beta}2}\int_0^1 w^{-\frac d2 - 1 - \frac 1{2(1-\beta)}} \exp\{-\Omega w - c_\beta w^{-\frac 1{1-\beta}}\}~\mathrm dw\\
\asymp & ~ C_{d, \beta} t^{-\frac{d\beta}2},\nonumber
\end{align}
due to the fast decay of $f_\beta$ in a neighbourhood of $0$. Thus combining the estimates for $I_1$ and $I_2$ gives~(\ref{eq:thm1case1}).

\underline{\bf Case 2: $\Omega \geq 1$.}
In this case we use the Laplace method as described in Appendix \ref{sec:asymptoticmethods}. Firstly for $I_1$, using $g(w) = w^{\frac d2 - 1}$, $h(w) = w$ and $b=1$ in (\ref{eq:laplaceboundary}) we have
\[I_1 = C t^{-\frac{d\beta}2}\int_1^\infty w^{\frac d2 - 1} \exp\{-\Omega w\}~\mathrm dw \sim t^{-\frac {d\beta}2}\Omega^{-1}\exp\{-\Omega\},\]
and in particular the estimate
\[I_1 \asymp C t^{-\frac{d\beta}2} \Omega^{-1}\exp\{-\Omega\}, \quad \Omega \geq 1.\]

For the second integral, we use Proposition \ref{prop:asymptoticcomputation} with $N = \frac d2 - 1 - \frac 1{2(1-\beta)}$ and $a = \frac 1{1-\beta}$,
\begin{align*}
I_2 = & ~C t^{-\frac{d\beta}2}\int_0^1 w^{-\frac d2 - 1 + \frac 1\beta} \exp\{-\Omega w\} f_\beta(w^{\frac 1\beta})~\mathrm dw\\
\sim&~ C t^{-\frac{d\beta}2}\Omega^{-\frac d2 \left(\frac{1-\beta}{2-\beta}\right)} \exp\{-C\Omega^{\frac 1{2-\beta}}\},
\end{align*}
and again in particular, the two-sided estimate
\[I_2 \asymp Ct^{-\frac{d\beta}2}\Omega^{-\frac d2 \left(\frac{1-\beta}{2-\beta}\right)} \exp\{-C\Omega^{\frac 1{2-\beta}}\}.\]

Combinging the estimates for $I_1$ and $I_2$ shows (\ref{eq:thm1case2}), and we are done.

\end{proof}

If one additionally assumes that the diffusion coefficients $A(x)$ of (\ref{eq:paradivform}) are twice continuously differentiable, the following estimates hold for the spatial derivatives of the fundamental solution of~(\ref{eq:paradivform}),
\be \label{eq:spatialaronsonestimate}\left|\frac{\partial}{\partial x}G(t,x,y)\right| \leq C t^{-\frac{d+1}2}\exp\left\{-C\frac{|x-y|^2}t\right\},\ee
for $(t, x, y) \in (0, \infty) \times \R^d \times \R^d$. We next have estimates for the spatial deriviatve of $G^{(\beta)}(t,x, y)$. 
\begin{prop}
Under the same assumptions as  Theorem \ref{thm:globaldivergenceestimate}, assume additionally that $A(x)$ is twice continuously differentiable, then the following estimates for the spatial derivatives of the Green's function $G^{(\beta)}(t, x, y)$ holds for all $(t, x, y) \in (0, \infty) \times\R^d \times\R^d$,
\begin{itemize}
\item
For $\Omega \leq 1$,
\[
\left|\pderiv{}{x}G^{(\beta)}(t, x, y)\right| \leq C\left\{
\begin{array}{lc}
t^{-\beta}(|\log \Omega| + 1),& d=1, \\
t^{-\frac{(d+1)\beta}2}\Omega^{1-\frac{d+1}2},& d \geq 2.
\end{array}
\right.
\]

\item For $\Omega \geq 1$,
\[\left|\pderiv{}{x}G^{(\beta)}(t, x, y)\right| \leq C t^{-\frac{(d+1)\beta}2} \Omega^{-\frac{(d+1)}2\left(\frac{1-\beta}{2-\beta}\right)}\exp\left\{-C_\beta \Omega^{-\frac 1{2-\beta}}\right\}.\]
\end{itemize}
\end{prop}
\begin{proof}
Recall that 
\[G^{(\beta)}(t, x, y) = \frac 1\beta \int_0^\infty G(t^\beta z, x, y) z^{-1 - \frac 1\beta}w_\beta(z^{-\frac 1\beta})~\mathrm dz,\]
where $G$ satisfies the global estimate (\ref{eq:spatialaronsonestimate}).
Using the triangle inequality after taking the derivative inside the integral,
\begin{align}\label{eqprf:glbldiffderiv}
\left|\pderiv{}{x}G^{(\beta)}(t, x, y)\right| =& \left|\pderiv{}{x}\frac 1\beta \int_0^\infty G(t^\beta z, x, y) z^{-1 - \frac 1\beta}w_\beta(z^{-\frac 1\beta})~\mathrm dz.\right|\nonumber\\
\leq&~C \int_0^\infty \left|\frac{\partial}{\partial x}G(t^\beta z,x,y)\right|z^{-1-\frac 1\beta} w_\beta(z^{-\frac 1\beta})~\mathrm dz\nonumber\\
\leq&~C t^{-\frac{(d+1)\beta}2}\int_0^\infty z^{-\frac{(d+1)}2}\exp\{-\Omega z^{-1}\}z^{-1-\frac 1\beta} w_\beta(z^{-\frac 1\beta})~\mathrm dz\nonumber\\
\leq&~C_\beta t^{-\frac{(d+1)\beta}2}\int_0^1 z^{-\frac{(d+1)}2}\exp\{-\Omega z^{-1}\}~\mathrm dz \\
&~+ C_\beta t^{-\frac{(d+1)\beta}2}\int_1^\infty z^{-\frac{d+1}2}\exp\{-\Omega z^{-1}\} z^{-1-\frac 1\beta} f_\beta(z^{-\frac 1\beta})~\mathrm dz\nonumber\\
=&~ C_\beta t^{-\frac{(d+1)\beta}2}\int_1^\infty w^{\frac{d+1}2 - 2} \exp\{-\Omega w\}~\mathrm dw\nonumber\\
&~+ C_\beta t^{-\frac{(d+1)\beta}2}\int_0^1 w^{\frac{d+1}2- 1 -\frac 1{2(1-\beta)}}\exp\{-\Omega w - c_\beta w^{-\frac 1{1-\beta}}\}~\mathrm dw\nonumber\\
:=&~I_1 + I_2,\nonumber
\end{align}
where in the above calculations, after using the estimates (\ref{eq:spatialaronsonestimate}) and (\ref{prop:asymptoticsstable}), we made the substitution $z = w^{-1}$. Note that the integrals $I_1$ and $I_2$ differ from those appearing in (\ref{eqprfdef:diffI1I2}) only by replacing $d$ with $d+1$. Thus the only change in the calculations is where the dimension dictates the behaviour of the estimate, namely in the integral $I_1$ under the regime $\Omega \leq 1$. In this case, make the substitution $w \Omega = V$,
\[
I_1 = C t^{-\frac{(d+1)\beta}2}\Omega^{1-\frac{d+1}2}\int_\Omega^\infty V^{\frac{d+1}2 - 2}\exp\{-V\}~\mathrm dV.
\]

For $d = 1$, we are in the same situation as (\ref{eqprf:i1log}), thus
\[
I_1 \sim Ct^{-\beta}(|\log\Omega| + 1), \quad \Omega \rightarrow 0,
\]
and in particular
\[
I_1 \leq Ct^{-\beta}(|\log\Omega| + 1), \quad \text{ for } \Omega \leq 1.
\]

Otherwise for $d \geq 2$ we have
\[
I_1 \leq C t^{-\frac{(d+1)\beta}2}\Omega^{1-\frac{d+1}2}\Gamma\left(\frac{d+1}2 - 1\right) = C_d t^{-\frac{(d+1)\beta}2}\Omega^{1-\frac{d+1}2}
\]

For the integral $I_2$, replacing $d$ with $d+1$ in (\ref{eqprf:diffI2moment}) does not spoil the estimate, thus
\[
I_2 \leq C_{d, \beta} t^{-\frac{(d+1)\beta}2}, \quad \text{ for } \Omega \leq 1.
\]

This shows
\[
\left|\pderiv{}{x}G^{(\beta)}(t, x, y)\right| \leq C\left\{
\begin{array}{lc}
t^{-\beta}(|\log \Omega| + 1),& d=1, \\
t^{-\frac{(d+1)\beta}2}\Omega^{1-\frac{d+1}2},& d \geq 2.
\end{array}
\right.
\]
for $\Omega \leq 1$ as required. For $\Omega \geq 1$, the estimates follow again by using the Laplace method. Namely taking $g(w) = w^{\frac{d+1}2-1}$, $h(w) = w$ and $b=1$ in (\ref{eq:laplaceboundary}) we have
\[
I_1 \leq t^{-\frac{(d+1)\beta}2} \Omega^{-1}\exp\left\{-\Omega\right\}, \quad \text{ for } \Omega \geq 1.
\]

Finally using $N = \frac{d+1}2 - 1- \frac 1{2(1-\beta)}$ and $a = \frac 1{1-\beta}$ in Proposition \ref{prop:asymptoticcomputation} gives us
\[
I_2 \leq C t^{-\frac{(d+1)\beta}2}\Omega^{-\frac{d+1}2\left(\frac{1-\beta}{2-\beta}\right)}\exp\left\{-C \Omega^{\frac 1{2-\beta}}\right\}, \quad \text{ for } \Omega \geq 1.
\]

Thus
\begin{align*}
\left|\pderiv{}{x}G^{(\beta)}(t, x, y)\right| & \leq I_1 + I_2\\
& \leq C t^{-\frac{(d+1)\beta}2}\Omega^{-\frac{d+1}2 \left(\frac{1-\beta}{2-\beta}\right)}\exp\{-C \Omega^{\frac 1{2-\beta}}\},
\end{align*}
as required. 
\end{proof}

\subsection{Time-Fractional Pseudo-Differential Evolution: Constant Coefficients}
Next we turn our attention to another class of problems, where the spatial operator is a homogeneous (constant coefficient) pseudo-differential operator. That is, for $\beta \in (0, 1)$ and $\alpha > 0$, 
\be\label{eq:fractionalevolstable}
D^\beta_0 u(t, x) = \Psi_\alpha(-i\nabla)u(t, x) , \quad u(0, x) = Y(x),
\ee
where $\Psi_\alpha$ is a pseudo-differential operator whose symbol is of the form
\[
\psi_\alpha(p) = -|p|^\alpha w_\mu(p/|p|),
\]
where $w_\mu$ is a positive function on $\mathbb S^{d-1}$, see (\ref{eq:defsymbolgeneral}). To this end, we use known properties of the Green's function $G_{\psi_\alpha}(t, x)$ of the evolution
\be
\label{eq:stableevolutionequation}\partial_t u = \Psi_{\alpha}(-i\nabla)u.
\ee

See for example \cite[Theorem 4.5.1]{kolokoltsov2019differential} or \cite{eidelman2004analytic} for the following estimates. Assuming that
\begin{itemize}
\item The function $w_\mu$ belongs to $C^{d+1+[\alpha]}(\mathbb S^{d-1})$.
\item The spectral measure $\mu$ has a density which is strictly positive (see (\ref{eqdef:spectralmeasure})).
\item $\alpha \in (0, 2)$,
\end{itemize}
then the Green's function $G_{\psi_\alpha}(t, x - y)$ of the evolution (\ref{eq:stableevolutionequation}) satisifies the following two-sided estimates for $(t, x, y) \in (0, \infty)\times \R^d \times \R^d$,
\be\label{eq:greensstable}
G_{\psi_\alpha}(t, x - y) \asymp C\min\left(\frac t{|x-y|^{d+\alpha}}, t^{-\frac d\alpha}\right).
\ee
Note that the restriction $\alpha \in (0, 2)$ and the positivity of the density of the spectral measure is required for the lower bound of $G_{\psi_\alpha}$ - the upper bound is still seen if we drop the strict positivity of the density $\mu$ and take any $\alpha > 0$. 

If additionally $w$ is $(d+1+[\alpha]+l)$-times continuously differentiable, then $G_{\psi_\alpha}(t, x)$ is $l$-times continuously differentiable in $x$ and for $(t, x,y)\in (0, \infty)\times \R^d\times \R^d$,
\be\label{eq:gstablederivative}
\left|\pderiv{^k}{x_{i_1}\cdots \partial x_{i_k}}G_{\psi_\alpha}(t,x-y)\right| \leq C \min \left(t^{-(d + k)/\alpha}, \frac{t}{|x-y|^{d+\alpha + k}}\right),
\ee
for all $k \leq l$ and $i_1, \cdots, i_k$. As discussed in the introduction, the solution of (\ref{eq:fractionalevolstable}) is given by
\[u(t,x) = E_{\beta}(\psi(-i\nabla)t^\beta)Y(x),\]
where
\[E_\beta (s) = \frac{1}{\beta}\int_0^\infty e^{sz} z^{-1-\frac{1}{\beta}} w_\beta( z^{-1/\beta})~\mathrm dz.\]

Thus the corresponding Greens function $G^{(\beta)}_{\psi_\alpha}(t, x, y)$ of (\ref{eq:fractionalevolstable}) is given by
\[
G^{(\beta)}_{\psi_\alpha}(t, x, y) := \frac 1\beta \int_0^\infty G_{\psi_\alpha}(t^\beta z, x-y) z^{-1-\frac 1\beta}w_\beta(z^{-\frac 1\beta})~\mathrm dz.
\]

In keeping with the previous section, we denote $\Omega := |x-y|^\alpha t^{-\beta}$. We have the following two-sided estimate for the Green's function $G^{(\beta)}_{\psi_\alpha}(t, x, y)$.

\begin{thm}\label{thm:stableglobal}
Let $\alpha\in (0, 2)$ and $\beta \in (0, 1)$. Assume that $w\in C^{(d+1+[\alpha])}(\mathbb S^{d-1})$, and that $w \geq w_0 > 0$ for some constant $w_0$. Further assume that the spectral measure $\mu$ of the stable operator $\Psi_\alpha$ has a strictly positive density. Then there exists a constant $C > 0$ such that the Green's function for the fractional evolution Equation~(\ref{eq:fractionalevolstable}) satisfies the following two-sided estimates for $(t, x-y) \in (0, \infty) \times \R^d$.

\begin{itemize}
\item For $\Omega\leq 1$,
\be\label{eqthm:greenchangedstablecase1}
G_{\psi_\alpha}^{(\beta)}(t, x- y) \asymp C\left\{
\begin{array}{lc}
t^{-\frac{d\beta}\alpha}, & d<\alpha ,\\
t^{-\beta}(|\log\Omega| + 1),& d=\alpha,\\
t^{-\frac{d\beta}{\alpha}}\Omega^{1-\frac d\alpha}, & d > \alpha.
\end{array}
\right.
\ee

\item For $\Omega \geq 1$,
\be\label{eqthm:greenschangedstablecase2}
	G^{(\beta)}_{\psi_\alpha}(t,x-y) \asymp C t^{-\frac{d\beta}{\alpha}}\Omega^{-1-\frac d\alpha}.
\ee
\end{itemize}
\end{thm}
\begin{proof}
We begin by splitting up the stable density using (\ref{prop:asymptoticsstable}),
\[
G^{(\beta)}_{\psi_\alpha}(t^\beta z, x - y) \asymp c_\beta \int_0^1 G_{\psi_\alpha}(t, x-y)~\mathrm dz + c_\beta \int_1^\infty G_{\psi_\alpha}(t^\beta z, x- y) z^{-1-\frac 1\beta}f_\beta(z^{-\frac 1\beta})~\mathrm dz,
\]
where $f_\beta(z)$ is defined in (\ref{prop:asymptoticsstable}). Before using the estimates (\ref{eq:greensstable}) for $G_{\psi_\alpha}$ (with $t = t^\beta z$), note that using the notation $\Omega = |x-y|^\alpha t^{-\beta}$, we have
\be\label{eq:minOmegasplit}
\min\left(t^{-\frac {d\beta}\alpha} \Omega^{-1-\frac d\alpha}z, t^{-\frac {d\beta}\alpha} z^{-\frac d\alpha}\right) = \left\{
\begin{array}{lc}
t^{-\frac{d\beta}\alpha}\Omega^{-1-\frac d\alpha}z, & \text{ for } z < \Omega,\\
t^{-\frac{d\beta}\alpha}z^{-\frac d\alpha},& \text{ for } z \geq \Omega.
\end{array}
\right.
\ee

Thus we have,
\begin{align}\label{eqprf:stabledefi1i2}
G^{(\beta)}_{\psi_\alpha}(t, x - y) \asymp &~c \int_0^1 \min\left(t^{-\frac {d\beta}\alpha} \Omega^{-1-\frac d\alpha}z, t^{-\frac {d\beta}\alpha} z^{-\frac d\alpha}\right) ~\mathrm dz\nonumber\\
&~+ c\int_1^\infty \min\left(t^{-\frac {d\beta}\alpha} \Omega^{-1-\frac d\alpha}z, t^{-\frac {d\beta}\alpha} z^{-\frac d\alpha}\right) z^{-1 - \frac 1\beta}f_\beta(z^{-\frac 1\beta})~\mathrm dz\\
:=&~I_1 + I_2.\nonumber
\end{align}

Now we deal with two cases.

\underline{\bf Case 1: $\Omega \leq 1$.}
Using \eqref{eq:minOmegasplit}, in this case the intergral $I_1$ equals
\begin{align*}
I_1 = & ~c t^{-\frac{d\beta}\alpha}\Omega^{-1-\frac d\alpha}\int_0^\Omega z~\mathrm dz +ct^{-\frac{d\beta}\alpha}\int_\Omega^1 z^{-\frac d\alpha}~\mathrm dz\\
= & ~\frac c2 t^{-\frac{d\beta}\alpha}\Omega^{1-\frac d\alpha} + c t^{-\frac{d\beta}\alpha}\int_\Omega^1 z^{-\frac d\alpha}~\mathrm dz.
\end{align*}

Note that for $d = \alpha$, the integral over the interval $(\Omega, 1)$ is
\[
t^{-\beta}\int_\Omega^1 z^{-1} ~\mathrm dz = t^{-\beta}|\log \Omega|.
\]

On the other hand, for $d \neq \alpha$ we have
\begin{align*}
t^{-\frac {d\beta}\alpha} \int_{\Omega}^1 z^{-\frac d\alpha}~dz &= \frac{1}{1-\frac d\alpha}t^{-\frac{d\beta}\alpha}(1-\Omega^{1-\frac d\alpha})\\
& \asymp C\left\{
\begin{array}{lc}
t^{-\frac {d\beta}\alpha}, & d < \alpha,\\
t^{-\frac{d\beta}\alpha}\Omega^{1-\frac d\alpha}, & d > \alpha.
\end{array}
\right.
\end{align*}

Thus in this case we have,
\[
I_1 \asymp C\left\{
\begin{array}{lc}
t^{-\frac{d\beta}\alpha}, & d<\alpha ,\\
t^{-\beta}(|\log\Omega| + 1),& d=\alpha,\\
t^{-\frac{d\beta}{\alpha}}\Omega^{1-\frac d\alpha}, & d > \alpha.
\end{array}
\right.
\]

Turning to $I_2$, note that the integral does not involve $\Omega$, and is convergent since $f_\beta(z^{-\frac 1\beta})$ is bounded and vanishes as $z \rightarrow \infty$. Thus
\[
I_2 = C t^{-\frac{d\beta}\alpha}\int_1^\infty z^{-\frac d\alpha - 1 - \frac1 \beta}f_\beta(z^{-\frac 1\beta})~\mathrm dz \asymp C_{\beta, d, \alpha} t^{-\frac{d\beta}\alpha}.
\]

Combining the estimates for $I_1$ and $I_2$ shows (\ref{eqthm:greenchangedstablecase1}).

\underline{\bf Case 2: $\Omega \geq 1$.}
In this case, the integral $I_1$ is simply
\[
I_1 = c t^{-\frac{d\beta}\alpha}\Omega^{-1-\frac d\alpha} \int_0^1 z~\mathrm dz  = \frac c2 t^{-\frac{d\beta}\alpha}\Omega^{-1-\frac d\alpha}.
\]

For the second integral, we have
\[
I_2 = c t^{-\frac{d\beta}\alpha}\Omega^{-1-\frac d\alpha}\int_1^\Omega z^{-\frac 1\beta}f_\beta(z^{-\frac 1\beta})~\mathrm dz + ct^{-\frac{d\beta}\alpha}\int_\Omega^{\infty} z^{-\frac d\alpha -1-\frac 1\beta}f_\beta(z^{-\frac 1\beta})~\mathrm dz.
\]

Note that the integral in the first term approaches a convergent integral (for large $\Omega$), while the second can be dealt with by using the Laplace method, see \ref{eq:laplaceboundary},
\begin{align*}
	I_2 \asymp &~c t^{-\frac{d\beta}\alpha}\Omega^{-1-\frac d\alpha} \int_1^\infty z^{-\frac 1\beta}f_\beta(z^{-\frac 1\beta})~\mathrm dz + ct^{-\frac{d\beta}\alpha}\int_{\Omega}^\infty z^{-\frac d\alpha - 1+\frac 1{2(1-\beta)}}\exp\{-c_\beta z^{\frac 1{1-\beta}}\}~\mathrm dz\\
	\asymp &~C_\beta t^{-\frac {d\beta}\alpha} \Omega^{-1-\frac d\alpha} + ct^{-\frac{d\beta}\alpha}\Omega^{-\frac d\alpha - \frac 1{2(1-\beta)}}\exp\{-c \Omega^{\frac 1{1-\beta}}\}\\
	\asymp &~C t^{-\frac {d\beta}\alpha} \Omega^{-1-\frac d\alpha},
\end{align*}
where we have used (\ref{eq:laplaceboundary}) with $g(x) = x^{-\frac d\alpha -1+ \frac 1{2(1-\beta)}}$ and $h(x) = x^{\frac 1{1-\beta}}$.
Combining the estimates for $I_1$ and $I_2$ proves (\ref{eqthm:greenschangedstablecase2}), which completes the proof.
\end{proof}

\begin{prop}
Under the assumptions of  Theorem \ref{thm:stableglobal}, assume additionally that $w \in C^{d + 1+[\alpha] + l}(\mathbb S^{d-1})$. Then the following estimates hold,
\begin{itemize}
\item For $\Omega\leq 1$,
\be\label{eqthm:stblderivcase1}
\left|\pderiv{^k}{x_{i_1}\cdots \partial x_{i_k}}G_{\psi_\alpha}^{(\beta)}(t, x- y)\right| \leq C\left\{
\begin{array}{lc}
t^{-\frac{(d+k)\beta}\alpha} & d+k<\alpha,\\
t^{-\beta}(|\log(\Omega)| + 1)& d+k=\alpha,\\
t^{-\frac{(d+k)\beta}\alpha}\Omega^{1-\frac{(d+k)}\alpha} & d+k>\alpha,
\end{array}
\right.
\ee
for all $k \leq l$ and $i_1, \cdots, i_k$.
\item For $\Omega \geq 1$,
\be\label{eqthm:stblderivcase2}
	\left|\pderiv{^k}{x_{i_1}\cdots\partial x_{i_k}}G_{\psi_\alpha}^{(\beta)}(t, x-y)\right| \leq Ct^{-\frac{(d+k)\beta}{\alpha}}\Omega^{-1-\frac{(d+k)}\alpha},
\ee
for all $k \leq l$ and $i_1, \cdots, i_k$.
\end{itemize}
\end{prop}
\begin{proof}
Using (\ref{prop:asymptoticsstable}) and (\ref{eq:gstablederivative}),
\begin{align*}
\left|\pderiv{^k}{x_{i_1}\cdots\partial x_{i_k}}G_{\psi_\alpha}^{(\beta)}(t, x-y)\right| \leq &~C_{k, d,\alpha, \beta}\int_0^1 \left|\pderiv{^k}{x_{i_1}\cdots\partial x_{i_k}}G_{\psi_\alpha}(t^\beta z, x-y)\right|~\mathrm dz\\
+& ~C\int_1^\infty \left|\pderiv{^k}{x_{i_1}\cdots\partial x_{i_k}}G_{\psi_\alpha}(t^\beta z, x-y)\right| z^{-1-\frac 1\beta} f_\beta(z^{-\frac 1\beta})~\mathrm dz\\
\leq &~I_1 + I_2,
\end{align*}
where
\[
I_1 := C \int_0^1 \min\left(t^{-\frac {(d+k)\beta}\alpha} \Omega^{-1-\frac{d+k}\alpha}z, t^{-\frac {(d+k)\beta}\alpha} z^{-\frac {d+k}\alpha}\right)~\mathrm dz,
\]
and
\be\label{eqprfdef:stablederivi1}
I_2 := C \int_1^\infty \min\left(t^{-\frac {(d+k)\beta}\alpha} \Omega^{-1-\frac{d+k}\alpha}z, t^{-\frac {(d+k)\beta}\alpha} z^{-\frac {d+k}\alpha}\right) z^{-1-\frac 1\beta} f_\beta(z^{-\frac 1\beta})~\mathrm dz.
\ee

Again we are in the situation where these integrals are the same as those found in (\ref{eqprf:stabledefi1i2}) after changing $d \mapsto d + k$. Thus we have for $\Omega \geq 1$,
\[
I_1 = \frac c2 t^{-\frac{(d+k)\beta}\alpha}\Omega^{-1-\frac{d+k}\alpha},
\]
and
\begin{align*}
I_2 = & ~c t^{-\frac{(d+k)\beta}\alpha}\Omega^{-1-\frac{d+k}\alpha}\int_1^\Omega z^{-\frac 1\beta}f_\beta(z^{-\frac 1\beta})~\mathrm dz + ct^{-\frac{(d+k)\beta}\alpha}\int_\Omega^{\infty} z^{-\frac{d+k}\alpha -1-\frac 1\beta}f_\beta(z^{-\frac 1\beta})~\mathrm dz\\
\leq & ~ C t^{-\frac{(d+k)\beta}\alpha} \Omega^{-1-\frac{d+k}\alpha} + c t^{-\frac{(d+k)\beta}\alpha} \Omega^{-\frac{d+k}\alpha - \frac 1{2(1-\beta)}}\exp\{-c\Omega^{\frac 1{1-\beta}}\}\\
\leq & ~ C_{d, k, \alpha, \beta}  t^{-\frac{(d+k)\beta}\alpha} \Omega^{-1-\frac{d+k}\alpha}.
\end{align*}

Combining the estimates for $I_1$ and $I_2$ gives us (\ref{eqthm:stblderivcase2}).
For $\Omega \leq 1$ we have
\[
I_ 2 = c t^{-\frac{(d+k)\beta}\alpha}\int_1^\infty z^{-\frac{d+k}\alpha - 1 - \frac 1\beta}f_\beta(z^{-\frac 1\beta})~\mathrm dz \leq C t^{-\frac{(d+k)\beta}\alpha}.
\]

It only remains to check the estimate for $I_1$ when $\Omega \leq 1$,
\begin{align*}
I_1 & = c t^{-\frac{(d+k)\beta}\alpha} \Omega^{-1-\frac{d+k}\alpha} \int_0^\Omega z~ \mathrm dz + c t^{-\frac{(d+k)\beta}\alpha}\int_\Omega^1 z^{-\frac{d+k}\alpha}~\mathrm dz \\
& = \frac c2 t^{-\frac{(d+k)\beta}\alpha} \Omega^{1-\frac{d+k}\alpha} + c t^{-\frac{(d+k)\beta}\alpha}\int_\Omega^1 z^{-\frac{d+k}\alpha}~\mathrm dz.
\end{align*}

Thus we have
\[
I_1 \leq  C\left\{
\begin{array}{lc}
t^{-\frac{(d+k)\beta}\alpha}, & d+k <\alpha,\\
t^{-\beta}(|\log\Omega| + 1),& d + k=\alpha ,\\
t^{-\frac{(d+k)\beta}{\alpha}}\Omega^{1-\frac{(d+k)}\alpha}, & d + k > \alpha.
\end{array}
\right.
\]

Combining the estimates for $I_1$ and $I_2$ for $\Omega \leq 1$ gives us (\ref{eqthm:stblderivcase1}).
\end{proof}

\section{Local Estimates}\label{sec:localestimates}
In the following two sections we look at two other families of spatial operators which extend the global estimates obtained in the previous sections. Firstly we consider a more general second order elliptic operator (not necessarily in divergence form), then we consider homogeneous pseudo-differential operators with variable coefficients. In both cases we provide local (i.e., small-time) two-sided estimates for the Green's functions of the associated fractional evolution equations. The key point here is that for these spatial operators, we no longer have global (in time) estimates for the associated Green's functions. Before going to the new estimates, we describe how one turns local estimates into global estimates.
If for some Green's functions $G_0(t, x, y), G_1(t, x,y)$, one has the local two-sided estimate for some constant $c > 0$
\[\frac{1}{c}G_1(t,x,y) \leq G_0(t,x,y) \leq c G_1(t,x,y), \quad (t,x,y)\in (0, T] \times \R^d\times \R^d,\]
for some fixed $T > 0$, then by taking convolutions and using the Chapman-Kolmogorov equations,
\begin{align*}
G_0(2t, x, y) &= \int_{\R^d} G_0(t, x, z)G_0(t, z, y)~dz\\
& \leq \int_{\R^d}cG_1(t, x, z)cG_1(t, z, y)~dz\\
& = c^2 G_1(2t, x, y).
\end{align*}

Repeating this procedure $n$-times,
\begin{align*}
G_0(nt, x, y) &= \int_{\R^d}\cdots \int_{\R^d} G_0(t, x, x_1)\cdots G_0(t, x_{n}, y)dx_1\cdots dx_{n}\\
& \leq \int_{\R^d}\cdots \int_{\R^d} cG_1(t, x, x_1)\cdots cG_1(t, x_{n}, y)dx_1\cdots dx_{n}\\
& = c^n G_1(nt, x, z).
\end{align*}

By fixing $t$ and setting $\tau = nt$ (so that $\tau \approx n$ for large values of $n$ and $\tau$), we then get
\begin{align*}
G_0(\tau, x, y) &\leq c^{\tau/t}G_1(\tau, x, y)\\
&=e^{\frac{\tau}t \log c} G_1(\tau, x, y)\\
&\approx e^{\tau \tilde{c}}G_1(\tau, x, y), \quad \forall \tau > 0, x, y \in \R^d
\end{align*}

Applying the same procedure to the lower bound gives us the global two-sided estimate 
\be\label{eq:tricklocal-to-global}
e^{-c\tau} G_1(\tau,x,y) \leq G_0(\tau,x,y) \leq e^{c\tau} G_1(\tau,x,y)
\ee
for all $(\tau, x, y) \in (0, \infty) \times \R^d \times \R^d$.

\subsection{Time-Fractional Diffusion Equation: General Non-Degenerate}
In Section \ref{sec:aronsoncase} we derived global two-sided estimates for the Green's function of fractional evolution equations involving a fractional derivative in time and a second order elliptic operator \emph{in divergence form} as the spatial operator. The key point in that case is that Aronsons estimates provides two-sided Gaussian estimates that hold globally for all time $t > 0$. In this section we consider the case that the spatial operator is any non-degenerate diffusion operator, which can generally be of the form
\be\label{eq:Lgeneraloperator}
Lu(t, x) := a_{ij}(x) \partial_{x_i}\partial_{x_j} u(t, x) + b_i(x) \partial_{x_i}u(t, x) + c(x) u(t, x)\ee

Assuming that $a(x)$ is uniformly elliptic and continuously differentiable, $b(x)$ and $c(x)$ are continuous, and the uniform bound holds:
\[\sup_x \max(|\nabla a(x)| , |b(x)| , |c(x)|) \leq M.\]

Then the Green's function associated with (\ref{eq:Lgeneraloperator}), satisfies the following local estimates:
\be\label{eq:localaronsonestimate}
\frac{t^{-\frac d2}}{C}\exp\left\{-C\frac{|x-y|^2}{t}\right\} \leq G(t, x, y) \leq C t^{-\frac d2}\exp\left\{-C\frac{|x-y|^2}{t}\right\},
\ee
for $(t, x, y) \in (0, T) \times \R^d\times \R^d$ for some fixed $T > 0$. We also have the following estimates for the spatial derivative of the Green's function $G$,
\be\label{eq:localaronsonderivest}\left|\pderiv{}{x}G(t,x,y)\right| \leq C t^{-\frac{d+1}2}\exp\left\{-C \frac{|x-y|^2}t\right\},
\ee
for $(t,x,y) \in (0, T)\times \R^d \times \R^d$.
The main obstacle now is that the estimates for the Green's function of (\ref{eq:Lgeneraloperator}) are only for small-time, thus a serious problem seems to arise when trying to insert the local estimate into the Pollard-Zolotarev formula, which involves integrating over all time $z \in (0, \infty)$. However we use the trick described in the previous section to make the local estimates global, in (\ref{eq:tricklocal-to-global}). To this end, the following two-sided estimate holds for $G(t, x, y)$ for all $(\tau, x, y) \in (0, \infty) \times \R^d\times \R^d$,
\be\label{eq:globalaronson}
e^{-c\tau}\tau^{-\frac{d}{2}}\exp\left\{-C\frac{|x-y|^2}{\tau}\right\}\leq G(\tau, x, y) \leq e^{c\tau}\tau^{-\frac{d}{2}}\exp\left\{-C\frac{|x-y|^2}{\tau}\right\},
\ee
for some constant $c$. In addition, for all $(\tau, x, y) \in (0, \infty) \times \R^d \times \R^d$,
\[
\left|\pderiv{}{x} G(\tau, x, y)\right| \leq  e^{c\tau}\max(\tau^{-\frac 12},1)\tau^{-\frac d2}\exp\left\{-C\frac{|x-y|^2}\tau\right\}.
\]

Alternatively we can split the estimates for the spatial derivative up into small-time and large-time - for $\tau \in (0, 1)$,
\be\label{eq:smallaronsonderiv}
\left|\pderiv{}{x} G(\tau, x, y)\right| \leq  C \tau^{-\frac{d+1}2}\exp\left\{-C\frac{|x-y|^2}\tau\right\},
\ee
and for $\tau \in (1, \infty)$,
\be\label{eq:largearonsonderiv}
\left|\pderiv{}{x} G(\tau, x, y)\right| \leq e^{c\tau}\tau^{-\frac d2}\exp\left\{-C\frac{|x-y|^2}\tau\right\}.
\ee

Now we proceed to obtain estimates for the Green's function of the fractional evolution
\[D^{\beta}_0 u(t, x) = Lu(t,x),\]
where $L$ is defined as above in (\ref{eq:Lgeneraloperator}). The Green's function for this fractional evolution equation is given~by
\be\label{eq:Green'slocaldiff}G^{(\beta)}(t, x, y) = \frac{1}{\beta}\int_0^\infty G(t^\beta z, x, y)z^{-1-\frac{1}{\beta}} w_\beta( z^{-\frac{1}{\beta}})~\mathrm dz.\ee

Again let $\Omega = |x-y|^2 t^{-\beta}$. We have the following local estimates for the Green's function $G^{(\beta)}$~above.
\begin{thm}\label{thm:fracdiffusionlocal}
Assume that $a(\cdot) \in C^1(\R^d)$ is uniformly elliptic and $b(\cdot),c(\cdot) \in C(\R^d)$. Suppose also that,
\[\sup_x\max (|\nabla a(x)|, |b(x)|, |c(x)|) \leq M.\]

Then for a fixed $T> 0$, there exists constants $C_1, C_2, C_3$ such that for $(t, x, y)\in (0, T]\times\R^d\times \R^d$ the Green's function $G^{(\beta)}(t, x, y)$ defined by (\ref{eq:Green'slocaldiff}) satisfies the following estimates,
\begin{itemize}
\item For $\Omega \leq 1$,
\be\label{eqthm:lcldiffG1}
G^{(\beta)}(t, x, y)\asymp C_1\left\{
\begin{array}{lc}
t^{-\frac \beta 2}, & d = 1,\\
t^{-\beta}(|\log\Omega| + 1),& d=2,\\
t^{-\frac{d\beta}{2}}\Omega^{1-\frac d2}, & d \geq 3.
\end{array}
\right.
\ee

\item For $\Omega \geq 1$,
\be\label{eqthm:lcldiffG2}
G^{(\beta)}(t, x, y) \asymp C_2 t^{-\frac{d\beta}2}\Omega^{-\frac d2\left(\frac{1-\beta}{2-\beta}\right)}\exp\{-C_3\Omega^{\frac 1{2-\beta}}\},
\ee
where $C_1, C_2$ depends on $T, d, \beta$ and $C_3$ depends on $T$ and $\beta$.
\end{itemize}
\end{thm}
\begin{proof}
First splitting up to the stable density,
\begin{align*}
G^{(\beta)}(t, x, y) &\asymp C_\beta \int_0^1 G(t^\beta z , x, y)~dz + C_\beta \int_1^\infty G(t^\beta z, x, y)z^{-1-\frac 1\beta} f_\beta(z^{-\frac 1\beta})~\mathrm dz\\
& =: I_1 + I_2.
\end{align*}

Note that on using the estimate (\ref{eq:localaronsonestimate}) in $I_1$, we have the same integral of the same name appearing in (\ref{eqprfdef:diffI1I2}). Thus for $\Omega \leq 1$,

\begin{align}\label{eqprf:lcldiffi1c1}
I_1 = C_\beta \int_0^1 G(t^\beta z, x, y)~\mathrm dz \asymp &~ C_{T} t^{-\frac{d\beta}2} \int_0^1 z^{-\frac{d}2}\exp\{-\Omega z^{-1}\}~\mathrm dz\nonumber\\
\asymp &~C \left\{
\begin{array}{lc}
t^{-\frac \beta 2}, & d = 1,\\
t^{-\beta}(|\log\Omega| + 1),& d=2,\\
t^{-\frac{d\beta}{2}}\Omega^{1-\frac d2}, & d \geq 3.
\end{array}
\right.
\end{align}

In addition for $\Omega \geq 1$,
\be\label{eqprof:diflocsmall}I_1 \asymp C_{T} t^{-\frac{d\beta}2} \Omega^{-1}\exp\{-\Omega\}.\ee

Turning our attention to $I_2$, let us consider seperately the upper and lower bound.
\\\underline{\bf Upper bound for $I_2$}\\
First applying the upper bound from (\ref{eq:globalaronson}) to $G$,
\begin{align}\label{eqprf:diffglbdefi2}
I_2 &\leq C t^{-\frac{d\beta}2}\int_1^\infty z^{-\frac d2 -1 - \frac 1\beta}\exp\{ct^\beta z - \Omega z^{-1}\}f_\beta(z^{-\frac 1\beta})~\mathrm dz\nonumber\\
& = C t^{-\frac{d\beta}2}\int_1^\infty z^{-\frac d2 -1 +\frac 1{2(1-\beta)}}\exp\{ct^\beta z - \Omega z^{-1}-c_\beta z^{\frac 1{1-\beta}}\}~\mathrm dz.
\end{align}

For $\Omega \leq 1$, we have
\begin{align*}
I_2 \leq &~C t^{-\frac{d\beta}2}\int_1^\infty z^{-\frac d2 - 1 - \frac 1{2(1-\beta)}} \exp\{ct^\beta z - c_\beta z^{\frac 1{1-\beta}}\}~\mathrm dz\\
\leq &~C t^{-\frac{d\beta}2}\int_1^\infty z^{-\frac d2 - 1 - \frac 1{2(1-\beta)}} \exp\{cT^\beta z - c_\beta z^{\frac 1{1-\beta}}\}~\mathrm dz\\
= & ~C_{T, d, \beta} t^{-\frac{d\beta}2},
\end{align*}
for $t < T$ for some fixed $T > 0$. Combining this with (\ref{eqprf:lcldiffi1c1}) gives (\ref{eqthm:lcldiffG1}). 

For $\Omega \geq 1$, we use again that the decay of $\exp\{-c_\beta z^{1/(1-\beta)}\}$ for large $z$ is stronger than the growth of $\exp\{ct^\beta z\}$ for large $z$ as long as $t < T$ for some fixed $T > 0$. That is,
\begin{align*}
I_2 &\leq C t^{-\frac{d\beta}2} \int_1^\infty z^{-\frac{d}2 - 1 + \frac 1{2(1-\beta)}}\exp\left\{-\Omega z^{-1}+ct^\beta z - c_\beta z^{\frac 1{1-\beta}}\right\}~\mathrm dz\nonumber\\
& \leq C t^{-\frac{d\beta}2} \int_1^\infty z^{-\frac d2 - 1 + \frac 1{2(1-\beta)}}\exp\left\{-\Omega z^{-1} - C_{T, \beta} z^{\frac 1{1-\beta}} \right\}~\mathrm dz\nonumber\\
& = C t^{-\frac{d\beta}2} \int_0^1 w^{\frac d2 - 1 -\frac 1{2(1-\beta)}} \exp\left\{-\Omega w - C_{T,\beta} w^{-\frac 1{1-\beta}}\right\}~\mathrm dw,
\end{align*}
where we have made the substitution $w = z^{-1}$ in the last line. Now we apply Proposition \ref{prop:asymptoticcomputation} with $N = \frac d2 - 1 - \frac 1{2(1-\beta)}$ and $a = \frac 1{1-\beta}$,
\[I_2 \leq C t^{-\frac{d\beta}2}\Omega^{-\frac d2 \left(\frac{1-\beta}{2-\beta}\right)}\exp\{-C \Omega^{\frac 1{2-\beta}}\}.\]

Note the constants in the above estimate depend on $T$. Combining this with (\ref{eqprof:diflocsmall}) gives us the required upper bound in (\ref{eqthm:lcldiffG2}).
\\\underline{\bf Lower bound for $I_2$}\\
Using the lower bound from (\ref{eq:globalaronson}) in $I_2$,
\[I_2 \geq c t^{-\frac{d\beta}2}\int_1^\infty z^{-\frac d2 - 1+\frac 1{2(1-\beta)}} \exp\{-c t^\beta z - \Omega z^{-1} - c_\beta z^{\frac 1{1-\beta}}\}~\mathrm dz.\]

Firstly for $\Omega \leq 1$,
\begin{align*}
I_2 &\geq C_\beta t^{-\frac{d\beta}2}\int_1^\infty z^{-\frac d2 - 1+ \frac 1{2(1-\beta)}} \exp\{-\Omega z^{-1} - c t^\beta z - c_\beta z^{\frac 1{1-\beta}}\}~\mathrm dz\nonumber\\
& \geq C_\beta t^{-\frac{d\beta}2}\int_1^\infty z^{-\frac d2 - 1+ \frac 1{2(1-\beta)}} \exp\{ - c t^\beta z - c_\beta z^{\frac 1{1-\beta}}\}~\mathrm dz\nonumber\\
& \geq C_{\beta, T} t^{-\frac{d\beta}2}\int_1^\infty z^{-\frac d2 - 1+ \frac 1{2(1-\beta)}} \exp\{ - cT^\beta z - c_\beta z^{\frac 1{1-\beta}}\}~\mathrm dz\nonumber\\
& = C_{T, \beta, d} t^{-\frac{d\beta}2}.
\end{align*}

Finally for $\Omega \geq 1$, 
\begin{align*}
I_2 & = C_\beta t^{-\frac{d\beta}2}\int_1^\infty z^{-\frac d2 - 1 + \frac 1{2(1-\beta)}} \exp\{-\Omega z^{-1} - c t^\beta z- c_\beta z^{\frac 1{1-\beta}}\}~\mathrm dz\\
& \geq C_\beta t^{-\frac{d\beta}2}\int_1^\infty z^{-\frac d2 - 1 + \frac{1}{2(1-\beta)}} \exp\{ - \Omega z^{-1} - (ct^\beta + c_\beta)z^{\frac 1{1-\beta}}\}~\mathrm dz\\
& \geq C_\beta t^{-\frac{d\beta}2}\int_1^\infty z^{-\frac d2 - 1 + \frac 1{2(1-\beta)}}\exp\{-\Omega z^{-1} - C_{T,\beta}z^{\frac 1{1-\beta}}\}~\mathrm dz,
\end{align*}
where we have used the fact that $\exp\{-ct^\beta z\} \geq \exp\{-ct^\beta z^{\frac 1{1-\beta}}\}$ for $z > 1$.
After making the substitution $z = w^{-1}$ we apply Proposition \ref{prop:asymptoticcomputation},
\begin{align*}
I_2 & \geq C_\beta t^{-\frac{d\beta}2} \int_0^1 w^{\frac d2 - 1 - \frac{1}{2(1-\beta)}} \exp\{-\Omega w - C_{T, \beta}w^{-\frac 1{1-\beta}}\}~\mathrm dw\\
&\geq C_1t^{-\frac{d\beta}2}\Omega^{-\frac d2 \left(\frac{1-\beta}{2-\beta}\right)}\exp\{-C_2 \Omega^{\frac 1{2-\beta}}\},
\end{align*}
where $C_1$ depends on $T, \beta$ and $d$, and $C_2$ depends on $T$ and $\beta$. Combining this with (\ref{eqprof:diflocsmall}) gives us the lower bound in (\ref{eqthm:lcldiffG2}), as required.
\end{proof}

Next we look at estimating the spatial derivative of the Green's function $G^{(\beta)}$, firstly for large-time using (\ref{eq:largearonsonderiv}) then for small-time using (\ref{eq:smallaronsonderiv}). As usual, let $\Omega := |x-y|^2 t^{-\beta}$. Firstly for large finite time,
\begin{prop}
Under the same assumptions as  Theorem \ref{thm:fracdiffusionlocal}, suppose further that $a(x)$ is twice continuously differentiable, and $b(x)$, $c(x)$ are continuously differentiable (with all derivatives bounded). Then for a fixed finite $T  > 1$, the following estimates hold for the spatial derivative of the Green's function $G^{(\beta)}(t, x, y)$ for $(t, x, y) \in (1, T) \times \R^d \times \R^d$, 
\begin{itemize}
\item
For $\Omega \leq 1$,
\be\label{eq:propderivgaussianlocal1}\left|\pderiv{}{x}G^{(\beta)}(t, x, y)\right| \leq C_{T, d, \beta}\left\{
\begin{array}{lc}
t^{-\beta}(|\log \Omega| + 1),& d = 1,\\
|x-y|^{1-d},& d \geq 2.
\end{array}
\right.
\ee

\item
For $\Omega \geq 1$,
\be
\label{eq:propderivgaussianlocal2}\left|\pderiv{}{x}G^{(\beta)}(t, x, y)\right| \leq C_{T, d, \beta} |x-y|^{-d\left(\frac{1-\beta}{2-\beta}\right)}\exp\{-C_{T, \beta} |x-y|^{\frac 2{2-\beta}}\}.
\ee
\end{itemize}
\end{prop}
\begin{proof}
We start as usual by first splitting up the integral into small and large $z$, and also use the triangle~inequality,
\begin{align*}
\left|\pderiv{}{x}G^{(\beta)}(t, x, y)\right|\leq C_{\beta} &\int_0^1 \left|\pderiv{}{x}G(t^\beta z, x, y)\right|~\mathrm dz\nonumber\\
+ &C_\beta \int_1^\infty \left|\pderiv{}{x}G(t^\beta z, x, y)\right| z^{-1-\frac 1\beta} f_\beta(z^{-\frac 1\beta})~\mathrm dz.
\end{align*}

Note that $t \in (1, T)$ means that $t^{-\beta} \in (T^{-\beta}, 1)$. Thus for $z \in (1, \infty)$ we have $z \geq t^{-\beta}$. Now we use the local estimate (\ref{eq:smallaronsonderiv}) for the first integral and (\ref{eq:largearonsonderiv}) for the second,
\begin{align*}
\left|\pderiv{}{x}G^{(\beta)}(t, x, y)\right| \leq &~C t^{-\frac{(d+1)\beta}2}\int_0^1 z^{-\frac{d+1}2}\exp\{-\Omega z^{-1}\}~\mathrm dz\\
&+ C t^{-\frac{d\beta}2}\int_1^\infty z^{-\frac d2 - 1 +\frac 1{2(1-\beta)}}\exp\left\{-\Omega z^{-1}+ct^\beta z - c_\beta z^{\frac 1{1-\beta}}\right\}~\mathrm dz\\
=: &~ I_1 + I_2.
\end{align*}

Note that the integral in $I_1$ is the same as (\ref{eqprf:glbldiffderiv}), and thus for $\Omega \leq 1$
\[I_1 = C t^{-\frac{(d+1)\beta}2} \int_0^1 z^{-\frac{d+1}2} \exp\{-\Omega z^{-1}\}\mathrm dz \leq \left\{
\begin{array}{lc}
t^{-\beta}(|\log \Omega| + 1),&d = 1,\\
t^{-\frac{(d+1)\beta}2}\Omega^{1-\frac{d+1}2},& d \geq 2.
\end{array}
\right.
\]

Note however that $t \in (1, T)$, which means that $t^{-\beta} \in (T^{-\beta}, 1)$. Thus
\be\label{eq:i1largefinitetimeondiagonal}
I_1 \leq C_{T, \beta, d}
\left\{
\begin{array}{lc}
t^{-\beta}(|\log \Omega| + 1),&d = 1,\\
|x-y|^{1-d},& d \geq 2.
\end{array}
\right.
\ee

For $\Omega \geq 1$,
\[I_1 \leq C t^{-\frac{(d+1)\beta}2}\Omega^{-1}\exp\{-\Omega\} \leq C_{T, d, \beta} |x-y|^{-2} \exp\{-C_{T,\beta} |x-y|^2\}.\]

As for the integral $I_2$, this is the same one which appeared in the previous proof, (\ref{eqprf:diffglbdefi2}), and thus for $\Omega \leq 1$,
\[I_2 \leq C t^{-\frac{d\beta}2} \leq C_{T, d, \beta}.\]

Combining this with (\ref{eq:i1largefinitetimeondiagonal}) which gives both (\ref{eq:propderivgaussianlocal1}).
Finally an application of Proposition \ref{prop:asymptoticcomputation} gives for $\Omega \geq 1$,
\begin{align*}
I_2 & = C t^{-\frac{d\beta}2}\int_1^\infty z^{-\frac d2 - 1 +\frac 1{2(1-\beta)}}\exp\left\{-\Omega z^{-1}+ct^\beta z - c_\beta z^{\frac 1{1-\beta}}\right\}~\mathrm dz\\
& \leq C t^{-\frac{d\beta}2}\int_1^\infty z^{-\frac d2 - 1 +\frac 1{2(1-\beta)}}\exp\left\{-\Omega z^{-1} - C_{T, \beta} z^{\frac 1{1-\beta}}\right\}~\mathrm dz\\
& \leq C t^{-\frac{d\beta}2} \Omega^{-\frac d2\left(\frac{1-\beta}{2-\beta}\right)}\exp\{-C \Omega^{\frac 1{2-\beta}}\}\\
& \leq C_{T, d, \beta} |x-y|^{-d\left(\frac{1-\beta}{2-\beta}\right)}\exp\{-C_{T, \beta} |x-y|^{\frac 2{2-\beta}}\}.
\end{align*}

Combining this with the estimate for $I_1$, gives the estimate (\ref{eq:propderivgaussianlocal2}) for $\Omega \geq 1$, as required.
\end{proof}

Next we have the estimates for small-time.
\begin{prop}
Under the same assumptions as  Theorem \ref{thm:fracdiffusionlocal}, suppose further that $a(x)$ is twice continuously differentiable, and $b(x)$, $c(x)$ are continuously differentiable (with all derivatives bounded). Then the following estimates hold for the spatial derivative of the Green's function $G^{(\beta)}(t, x, y)$ for $(t, x, y) \in (0, 1) \times \R^d \times \R^d$, 
\begin{itemize}
\item
For $\Omega \leq 1$,
\[\left|\pderiv{}{x}G^{(\beta)}(t, x, y)\right| \leq C_{d, \beta}\left\{
\begin{array}{lc}
t^{-\beta}(|\log \Omega| + 1),& d = 1,\\
t^{-\frac{(d+1)\beta}2}\Omega^{1-\frac{d+1}2},& d \geq 2.
\end{array}
\right.
\]

\item
For $1\leq \Omega \leq t^{-\beta\left(\frac{2-\beta}{1-\beta}\right)}$,
\[\left|\pderiv{}{x}G^{(\beta)}(t, x, y)\right| \leq C t^{-\frac{(d+1)\beta}2}\Omega^{-\left(\frac{d+1}2\right)\left(\frac{1-\beta}{2-\beta}\right)}\exp\{-C\Omega^{\frac 1{2-\beta}}\}.\]

\item
For $\Omega \geq t^{-\beta\left(\frac{2-\beta}{1-\beta}\right)}$,
\[\left|\pderiv{}{x}G^{(\beta)}(t, x, y)\right|\leq C t^{-\frac{d\beta}2}\Omega^{-\frac d2\left(\frac{1-\beta}{2-\beta}\right)} \exp\{-C\Omega^{\frac 1{2-\beta}}\}.\]

\end{itemize}
\end{prop}
\begin{proof}
Splitting the integral up using the stable density then using the estimates (\ref{eq:smallaronsonderiv}) and (\ref{eq:largearonsonderiv}), 
\begin{align*}
\left|\pderiv{}{x}G^{(\beta)}(t, x, y)\right| \leq &~C_{\beta} t^{-\frac{(d+1)\beta}2}\int_0^1 z^{-\frac{d+1}2}\exp\{-\Omega z^{-1}\}~\mathrm dz\\
&+ C_\beta \int_1^\infty \max((t^\beta z)^{-\frac{1}{2}}, 1)z^{-\frac d2 - 1-\frac 1\beta}\exp\left\{-\frac{\Omega}z+ct^\beta z \right\}f_\beta(z^{-\frac 1\beta})~\mathrm dz\\
=&~C_{\beta} t^{-\frac{(d+1)\beta}2}\int_0^1 z^{-\frac{d+1}2}\exp\{-\Omega z^{-1}\}~\mathrm dz\\
&+ C_\beta t^{-\frac{(d+1)\beta}2}\int_1^{t^{-\beta}}z^{-\frac{d+1}2 - 1 + \frac 1{2(1-\beta)}} \exp\left\{-\Omega z^{-1}+ct^\beta z- c_\beta z^{\frac 1{1-\beta}}\right\}~\mathrm dz\\
&+ C_\beta t^{-\frac{d\beta}2}\int_{t^{-\beta}}^\infty z^{-\frac d2 -1 + \frac 1{2(1-\beta)}} \exp\left\{-\Omega z^{-1}+ct^\beta z - c_\beta z^{\frac 1{1-\beta}}\right\}~\mathrm dz\\
=: &~ I_1 + I_2 + I_3.
\end{align*}

Now we investigate the usual cases.

\underline{\bf Case 1: $\Omega \leq 1$}

The integral in $I_1$, being the same as the one in (\ref{eq:i1largefinitetimeondiagonal}), has the upper bound
\[ I_1 = C t^{-\frac{(d+1)\beta}2} \int_0^1 z^{-\frac{d+1}2} \exp\{-\Omega z^{-1}\}~\mathrm dz \leq C\left\{
\begin{array}{lc}
t^{-\beta}(|\log \Omega| + 1),&d = 1,\\
t^{-\frac{(d+1)\beta}2}\Omega^{1-\frac{d+1}2},& d \geq 2.
\end{array}
\right.
\]

The other two integrals in $I_2$ and $I_3$ approach convergent integrals for bounded $\Omega$, so
\begin{align*}
I_2 =~&C t^{-\frac{(d+1)\beta}2} \int_1^{t^{-\beta}} z^{-\frac{d+1}2-1+\frac{1}{2(1-\beta)}} \exp\{-\Omega z^{-1} + ct^\beta z - c_\beta z^{\frac{1}{1-\beta}}\}~\mathrm dz\\
 \leq~&C t^{-\frac{(d+1)\beta}2} \int_1^{\infty} z^{-\frac{d+1}2-1+\frac{1}{2(1-\beta)}} \exp\{ct^\beta z - c_\beta z^{\frac{1}{1-\beta}}\}~\mathrm dz\\
 \leq~&C_{d, \beta} t^{-\frac{(d+1)\beta}2},
 \end{align*}
and
 \begin{align*}
 I_3 =&~C t^{-\frac{d\beta}2} \int_{t^{-\beta}}^\infty z^{-\frac d2 -1 + \frac 1{2(1-\beta)}} \exp\left\{-\Omega z^{-1}+ct^\beta z - c_\beta z^{\frac 1{1-\beta}}\right\}~\mathrm dz\\
 \leq &~ Ct^{-\frac{d\beta}2}\int_{t^{-\beta}}^\infty z^{-\frac d2 - 1 + \frac 1{2(1-\beta)}} \exp\left\{ct^\beta z - c_\beta z^{\frac 1{1-\beta}}\right\}~\mathrm dz\\
 \leq &~ t^{\beta-\frac{\beta}{2(1-\beta)}}\exp\{-C_\beta t^{-\frac \beta{1-\beta}}\}\\
 \leq &~C_{d, \beta} t^{-\frac{d\beta}2}.
  \end{align*}

Thus in this case,
\[\left|\pderiv{}{x}G^{(\beta)}(t, x,y)\right|\leq C\left\{
\begin{array}{lc}
t^{-\beta}(|\log \Omega| + 1),&d = 1,\\
t^{-\frac{(d+1)\beta}2}\Omega^{1-\frac{d+1}2},& d \geq 2.
\end{array}
\right.
\]

 \underline{\bf Case 2: $\Omega \geq 1$}
 
A direct application of the Laplace method gives
\[I_1 \leq t^{-\frac{(d+1)\beta}2}\Omega^{-1}\exp\{-\Omega\}.\]

For the second integral we have,
\begin{align*}
I_2\leq &~C t^{-\frac{(d+1)\beta}2}\int_1^\infty z^{-\frac d2 - 1+ \frac 1{2(1-\beta)}}\exp\{-\Omega z^{-1} - C_\beta z^{\frac 1{1-\beta}}\}~\mathrm dz \\
\leq &~C t^{-\frac{(d+1)\beta}2}\Omega^{-\left(\frac{d+1}2\right)\left(\frac{1-\beta}{2-\beta}\right)}\exp\{-C\Omega^{\frac 1{2-\beta}}\}
\end{align*}
where we have used Proposition \ref{prop:asymptoticcomputation} in the last estimate.
Finally since $t \in (0, 1)$, another application of Proposition \ref{prop:asymptoticcomputation} gives
 \begin{align*}
 I_3 &\leq Ct^{-\frac{d\beta}2}\int_1^\infty z^{-\frac d2 -1+\frac 1{2(1-\beta)}} \exp\{-\Omega z^{-1} - C z^{\frac 1{1-\beta}}\}~\mathrm dz \\
 & \leq C t^{-\frac{d\beta}2}\Omega^{-\frac d2\left(\frac{1-\beta}{2-\beta}\right)} \exp\{-C\Omega^{\frac 1{2-\beta}}\},
 \end{align*}

Note that
\[t^{-\frac{(d+1)\beta}2}\Omega^{-\left(\frac{d+1}2\right)\left(\frac{1-\beta}{2-\beta}\right)}\exp\{-C\Omega^{\frac 1{2-\beta}}\}  \leq C t^{-\frac{d\beta}2}\Omega^{-\frac d2\left(\frac{1-\beta}{2-\beta}\right)} \exp\{-C\Omega^{\frac 1{2-\beta}}\},\]
when $t^{-\frac \beta 2} \Omega^{-\frac 12 \left(\frac{1-\beta}{2-\beta}\right)}\leq 1$. Thus for $\Omega \geq t^{-\beta\left(\frac{2-\beta}{1-\beta}\right)}$,
\[\left|\pderiv{}{x}G^{(\beta)}(t, x,y)\right|\leq C t^{-\frac{d\beta}2}\Omega^{-\frac d2 \left(\frac{1-\beta}{2-\beta}\right)}\exp\{-C\Omega^{\frac 1{2-\beta}}\},\]
while for $1\leq \Omega \leq t^{-\beta\left(\frac{2-\beta}{1-\beta}\right)}$,
\[\left|\pderiv{}{x}G^{(\beta)}(t, x,y)\right|\leq C t^{-\frac{(d+1)\beta}2}\Omega^{-\frac{d+1}2\left(\frac{1-\beta}{2-\beta}\right)}\exp\{-C\Omega^{\frac 1{2-\beta}}\}.\]
\end{proof}

\subsection{Time-Fractional Pseudo-Differential Evolution: Variable Coefficients}

Finally we derive two-sided estimates for the Green's function of time-fractional stable-like equations. Stable-like operators are homogeneous pseudo-differential operators with variable coefficients (that depend on the spatial variable $x$, but not time). As noted earlier in (\ref{eq:greensstable}) the fundamental solution $G_\psi$ of the evolution equation
\[\partial_t u =-\psi_\alpha(-i \nabla)u,\]
with $\psi_\alpha(p) = |p|^\alpha w(p/|p|)$, satisfies the following two-sided estimate for all $(t, x - y) \in (0, \infty) \times \R^d$,  
\be\label{eq:stablegreenineq2}G_{\psi_\alpha}(t,x-y) \asymp C\min\left(\frac{t}{|x-y|^{d+\alpha}}, t^{-d/\alpha}\right).\ee

When the coefficients of the operator $\psi_\alpha$ depends also on the spatial variable, the same kind of estimates hold for small-time. Let $G_{\psi_\alpha, x}$ denote the fundamental solution to the pseudo-differential evolution equation
\[\partial_t u = -\psi_\alpha(x, -i\nabla)u,\]
with homogeneous symbol $\psi_\alpha(x, p) = |p|^\alpha w_\mu(x, p/|p|)$, where
\[w_\mu(x, p) = \int_{\mathbb S^{d-1}}|(p, s)|^\alpha \mu(x, \mathrm ds).\]

\begin{thm} Assume that $w_\mu(x, p)$ is a $\gamma$-Hölder continuous function in the variable $x$ taking values in a compact subset of $(0, \infty)$, $\gamma\in (0, 1]$. Assume further that $\mu$ has a strictly positive density. Then for some fixed $T>0$, there exists a constant $C > 0$ such that for $t\in (0,T)$ and $x, y \in \R^d$,
\[
\frac{1}{C}G_{\psi_\alpha}(t, x - y) \leq G_{\psi_\alpha, x}(t, x, y) \leq C G_{\psi_\alpha}(t, x - y)
\]
\end{thm}

What this means is that the global in time estimates (\ref{eq:stablegreenineq2}) for the Green's function $G_{\psi_\alpha}$, also serve as a small-time estimate for the Green's function $G_{\psi_\alpha, x}$. Indeed one would hope that operators with variable coefficients can be approximated by the method of freezing coefficients. So we have the following small-time estimate for $t \in (0, T)$, $x, y\in \R^d$
\be\label{eq:stablelikegreenineq}\frac 1C \min\left(\frac t{|x-y|^{d+\alpha}}, t^{-\frac d\alpha}\right) \leq G_{\psi_\alpha, x}(t,x,y) \leq C\min\left(\frac{t}{|x-y|^{d+\alpha}}, t^{-\frac d\alpha}\right),\ee
for some fixed $0< T < \infty$.
We also have the following estimates for the spatial derivatives of the $G_{\psi_\alpha, x}$, see \cite{kolokoltsov2019differential} (Theorem 5.8.3).
\begin{thm}\label{thm:localderivativesstablelike}
Let $\alpha > 0$, and denote by $l$ the maximal integer less than $\alpha$. Assume that $\mu \geq \mu_0 > 0$, for some positive number $\mu_0$, and $w_\mu(x, p)$ is $q$-times differentiable in $x$ and each of these derivatives be $(d+1 + (l+q)(\alpha + 1))$-times continuously differentiable in $p$ and all bounds uniform in $x, p$. Then for a fixed $T > 0$ and any $k\leq l$,
\be\label{eq:localderivativestablelike}
\left|\pderiv{^k}{x_{i_1}\cdots\partial x_{i_k}}G_{\psi_\alpha, x}(t, x, y)\right| \leq C \min\left(\frac{t}{|x-y|^{d+k+\alpha}},t^{-\frac{(d+k)}\alpha}\right)
\ee
for $(t, x, y) \in (0, T)\times \R^d \times \R^d$. 
\end{thm}
Using the same technique as the previous section to extend these small-time estimates to global estimates, we have the following two-sided estimates for $\tau > 0$, $x, y \in \R^d$,
\be\label{eq:globalstableliketwosided}
e^{-C\tau}\min\left(\frac{\tau}{|x-y|^{d+\alpha}},\tau^{-\frac d\alpha}\right) \leq G_{\psi_\alpha, x}(\tau,x, y) \leq e^{C\tau}\min\left(\frac{\tau}{|x-y|^{d+\alpha}}, \tau^{-\frac d\alpha}\right),
\ee
and
\be\label{eq:globalderivativestablelike}\left|\pderiv{^k}{x_{i_1}\cdots\partial x_{i_k}}G_{\psi_\alpha, x}(\tau, x, y)\right| \leq e^{C\tau}\max\left(\tau^{-\frac k \alpha}, 1\right) \min\left(\frac{\tau}{|x-y|^{d+\alpha}},\tau^{-\frac{d}\alpha}\right).
\ee

Now consider the following fractional evolution equation,
\begin{equation}\label{eq:fracevolvar}D^{\beta}_0u(t,x) =-\psi_\alpha(x, -i\nabla)u(t,x), \quad u(0, x) = Y(x),\end{equation}
with
\be\label{eq:defvariablesymbol}
\psi_\alpha(x, p) = |p|^\alpha w_\mu(x,p/|p|),
\ee
where $w_\mu$ satisfies the assumptions of  Theorem \ref{thm:localderivativesstablelike}.
The solution of (\ref{eq:fracevolvar}) is given by
\[u(t, x) = E_{\beta}(\psi_\alpha(x,-i\nabla)t^\beta)Y(x),\]
where $E_\beta$ is the Mittag-Leffler function. The Green's function of Equation (\ref{eq:fracevolvar}) is then
\be\label{eq:stablelikeGreen's} G^{(\beta)}_{\psi_{\alpha}, x}(t, x, y) = \frac{1}{\beta}\int_0^\infty G_{\psi_\alpha, x}(t^\beta z, x, y) z^{-1-\frac 1\beta}w_\beta(z^{-\frac1\beta})~\mathrm dz.\ee

Let $\Omega = |x-y|^\alpha t^{-\beta}$.
\begin{thm}\label{thm:fracevollocal}
Let $\alpha \in (0, 2)$ and $\beta \in (0, 1)$.
Assume that the function $w_\mu$ in (\ref{eq:defvariablesymbol}) is $\gamma$-Hölder continuous in the first variable and $k$-times continuously differentiable in the second variable. Assume further that the spectral measure $\mu$ has a strictly positive density. Then for a fixed $T> 0$ there exists constants $C$ such that for $(t, x, y)\in (0, T] \times \R^d \times \R^d$ the following two-sided estimates for (\ref{eq:stablelikeGreen's}) hold,
\begin{itemize}
\item For $\Omega \leq 1$,
\[
G^{(\beta)}_{\psi_\alpha, x}(t, x, y) \asymp ~C
\left\{
\begin{array}{lc}
t^{-\frac{d\beta}\alpha},& d < \alpha,\\
t^{-\beta}(|\log(\Omega)| + 1),& d=\alpha,\\
t^{-\frac{d\beta}\alpha}\Omega^{1-\frac d\alpha}, & d> \alpha.
\end{array}
\right.\\
\]

\item For $\Omega \geq 1$,
\[G^{(\beta)}_{\psi_\alpha, x}(t, x, y)\asymp~C t^{-\frac{d\beta}\alpha}\Omega^{-1-\frac d\alpha},\]
\end{itemize}
where the constants $C$ depend on $d, \alpha, \beta$ and $T$.
\end{thm}
\begin{proof}
We start by estimating the stable density with (\ref{prop:asymptoticsstable}),
\[
G^{(\beta)}_{\psi_\alpha, x} (t, x, y) \asymp C_\beta \int_0^1 G_{\psi_\alpha, x}(t^\beta z, x, y)~ dz + C_\beta \int_1^\infty G_{\psi_\alpha, x}(t^\beta z, x, y) z^{-1-\frac 1\beta} f_\beta(z^{-\frac 1\beta})~\mathrm dz.
\]

In the first integral, we use the estimate (\ref{eq:stablelikegreenineq}), and for the second term we use the global version~(\ref{eq:globalstableliketwosided}) with $\tau = t^\beta z$. Starting with the upper bound, we have
\begin{multline*}
G^{(\beta)}_{\psi_{\alpha}, x}(t, x, y) \leq C_T\int_0^1 \min\left(t^{-\frac{d\beta}\alpha}\Omega^{-1-\frac d\alpha}z, t^{-\frac{d\beta}\alpha}z^{-\frac d\alpha}\right)~\mathrm dz\\
~+ c\int_1^\infty \min\left(t^{-\frac{d\beta}\alpha}\Omega^{-1-\frac d\alpha}z, t^{-\frac{d\beta}\alpha}z^{-\frac d\alpha}\right)z^{-1-\frac 1\beta}e^{ct^\beta z}f_\beta(z^{-\frac 1\beta})~\mathrm dz.
\end{multline*}

Recall that, 
\[\min\left(t^{-\frac{d\beta}\alpha}\Omega^{-1- \frac d\alpha}z, t^{-\frac{d\beta}\alpha} z^{-\frac d\alpha}\right) = \left\{
\begin{array}{lc}
t^{-\frac{d\beta}\alpha}\Omega^{-1-\frac d\alpha}z,&  z \leq \Omega,\\
t^{-\frac{d\beta}\alpha}z^{-\frac d\alpha},& z \geq \Omega.
\end{array}
\right.\]

Then we have
\begin{align}\label{eqprfdef:stablelikeiup}
G^{(\beta)}_{\psi_\alpha, x}(t, x, y) \leq &~c \int_0^1 \min\left(t^{-\frac {d\beta}\alpha} \Omega^{-1-\frac d\alpha}z, t^{-\frac {d\beta}\alpha} z^{-\frac d\alpha}\right) ~\mathrm dz\nonumber\\
&~+ c\int_1^\infty \min\left(t^{-\frac {d\beta}\alpha} \Omega^{-1-\frac d\alpha}z, t^{-\frac {d\beta}\alpha} z^{-\frac d\alpha}\right) z^{-1 - \frac 1\beta}e^{ct^\beta z}f_\beta(z^{-\frac 1\beta})~\mathrm dz\\
:=&~I_1 + I^{up}_2,\nonumber
\end{align}
for the upper bound, and
\begin{align*}
G^{(\beta)}_{\psi_\alpha, x}(t, x, y) \geq &~\frac 1c \int_0^1 \min\left(t^{-\frac {d\beta}\alpha} \Omega^{-1-\frac d\alpha}z, t^{-\frac {d\beta}\alpha} z^{-\frac d\alpha}\right) ~\mathrm dz\\
&~+ C\int_1^\infty \min\left(t^{-\frac {d\beta}\alpha} \Omega^{-1-\frac d\alpha}z, t^{-\frac {d\beta}\alpha} z^{-\frac d\alpha}\right) z^{-1 - \frac 1\beta}e^{-ct^\beta z}f_\beta(z^{-\frac 1\beta})~\mathrm dz\\
:=&~I_1 + I^{lo}_2,
\end{align*}
for the lower bound.
Note that the integral in $I_1$ is the as the one appearing in (\ref{eqprf:stabledefi1i2}) and so, for $t < T$, 
\[I_1 \asymp C_T \left\{ \begin{array}{lc}
t^{-\frac{d\beta}\alpha},& d < \alpha,\\
t^{-\beta}(|\log\Omega|  + 1),& d = \alpha,\\
t^{-\frac{d\beta}\alpha}\Omega^{1-\frac d\alpha}, & d > \alpha,
\end{array}
\right.
\]
for $\Omega \leq 1$, and
\[I_1 \asymp t^{-\frac{d\beta}\alpha} \Omega^{-1-\frac d\alpha},\]
for $\Omega \geq 1$. For the remaining integral $I_2$, we have the usual two cases.

\underline{\bf Case 1: $\Omega \leq 1$.}
In this case we have
\[I^{up}_2 = C t^{-\frac{d\beta}\alpha}\int_1^\infty z^{-\frac d\alpha - 1 + \frac 1{2(1-\beta)}}\exp\{ct^\beta z - c_\beta z^{\frac 1{1-\beta}}\}~\mathrm dz.\]

This integral converges as long as $t < T$, since $\exp\{ct^\beta z\} \leq \exp\{C z^{\frac 1{1-\beta}}\}$ for sufficiently large $z$.~Thus
\[I^{up}_2 \leq C_{T, d, \beta, \alpha} t^{-\frac{d\beta}\alpha}.\]

On the other hand, we have
\[I^{lo}_2 = Ct^{-\frac{d\beta}\alpha}\int_1^\infty z^{-\frac d\alpha - 1 + \frac 1{2(1-\beta)}}\exp\{-ct^\beta z - c_\beta z^{\frac 1{1-\beta}}\}~\mathrm dz,\]
which is strictly positive for $t < T$, thus
\[I^{lo}_2 \geq C_{T, d, \beta, \alpha} t^{-\frac{d\beta}\alpha}.\]

Combining these estimates with those for $I_1$ gives the estimates for $G^{(\beta)}_{\psi_\alpha, x}$ for $\Omega \leq 1$.

\underline{\bf Case 2: $\Omega \geq 1$.}
In this case we have
\begin{multline*}
I^{up}_2 = C t^{-\frac{d\beta}\alpha}\Omega^{-1-\frac d\alpha}\int_1^\Omega z^{\frac 1{2(1-\beta)}} \exp\{ct^\beta z - z^{\frac 1{1-\beta}}\}~\mathrm dz \\
+ C t^{-\frac{d\beta}\alpha}\int_\Omega^\infty z^{-\frac d\alpha - 1 + \frac 1{2(1-\beta)}}\exp\{ct^\beta z- c_\beta z^{\frac 1{1-\beta}}\}~\mathrm dz, 
\end{multline*}
and
\begin{multline*}
I^{lo}_2 = C t^{-\frac{d\beta}\alpha}\Omega^{-1-\frac d\alpha}\int_1^\Omega z^{\frac 1{2(1-\beta)}} \exp\{-ct^\beta z - z^{\frac 1{1-\beta}}\}~\mathrm dz \\
+ C t^{-\frac{d\beta}\alpha}\int_\Omega^\infty z^{-\frac d\alpha - 1 + \frac 1{2(1-\beta)}}\exp\{-ct^\beta z- c_\beta z^{\frac 1{1-\beta}}\}~\mathrm dz.
\end{multline*}

Firstly we have,
\[t^{-\frac{d\beta}\alpha}\Omega^{-1-\frac d\alpha}\int_1^\Omega z^{\frac 1{2(1-\beta)}} \exp\{ct^\beta z - c_\beta z^{\frac 1{1-\beta}}\}~\mathrm dz \leq C_{T, d, \beta, \alpha} t^{-\frac{d\beta}\alpha}\Omega^{-1-\frac d\alpha}\]
and
\[t^{-\frac{d\beta}\alpha}\Omega^{-1-\frac d\alpha}\int_1^\Omega z^{\frac 1{2(1-\beta)}} \exp\{-ct^\beta z - c_\beta z^{\frac 1{1-\beta}}\}~\mathrm dz \geq \tilde{C}_{T, d, \beta, \alpha} t^{-\frac{d\beta}\alpha}\Omega^{-1-\frac d\alpha}.\]

Next note that $\exp\{t^\beta z\} \leq \exp\{T^\beta z\}$ and $\exp\{-t^\beta z\} \geq \exp\{-T^\beta z\}$ for $t < T$. Thus,
\begin{align}\label{eqprf:stbllikei2uplrge}
I^{up}_2 \leq &~C t^{-\frac{d\beta}\alpha} \Omega^{-1-\frac d\alpha} + C t^{-\frac{d\beta}\alpha} \int_\Omega^\infty z^{-\frac d\alpha -1+\frac 1{2(1-\beta)}} \exp\{t^\beta z - c_\beta z^{\frac 1{1-\beta}}\}~\mathrm dz \nonumber\\
\leq &~Ct^{-\frac{d\beta}\alpha}\Omega^{-1-\frac d\alpha} +C t^{-\frac{d\beta}\alpha}\int_\Omega^\infty z^{-\frac d\alpha - 1 + \frac 1{2(1-\beta)}}\exp\{-C_{T, \beta}z^{\frac 1{1-\beta}}\}~\mathrm dz\nonumber\\
\leq & ~Ct^{-\frac{d\beta}\alpha}\Omega^{-1-\frac d\alpha}  + Ct^{-\frac{d\beta}\alpha}\Omega^{-\frac d\alpha - \frac 1{2(1-\beta)}}\exp\{-C_{T, \beta} \Omega^{\frac 1{1-\beta}}\}\nonumber\\
\leq & ~C_{T, \beta, \alpha, d}t^{-\frac{d\beta}\alpha}\Omega^{-1-\frac d\alpha}
\end{align}
and
\begin{align*}
I^{lo}_2 \geq &~C t^{-\frac{d\beta}\alpha} \Omega^{-1-\frac d\alpha} + C t^{-\frac{d\beta}\alpha} \int_\Omega^\infty z^{-\frac d\alpha -1+\frac 1{2(1-\beta)}} \exp\{-t^\beta z - c_\beta z^{\frac 1{1-\beta}}\}~\mathrm dz \\
\geq &~Ct^{-\frac{d\beta}\alpha}\Omega^{-1-\frac d\alpha} +C t^{-\frac{d\beta}\alpha}\int_\Omega^\infty z^{-\frac d\alpha - 1 + \frac 1{2(1-\beta)}}\exp\{-C_{T, \beta}z^{\frac 1{1-\beta}}\}~\mathrm dz\\
\geq &~Ct^{-\frac{d\beta}\alpha}\Omega^{-1-\frac d\alpha}  + Ct^{-\frac{d\beta}\alpha}\Omega^{-\frac d\alpha - \frac 1{2(1-\beta)}}\exp\{-C_{T, \beta} \Omega^{\frac 1{1-\beta}}\}\\
\geq & ~C_{T, \beta, \alpha, d}t^{-\frac{d\beta}\alpha}\Omega^{-1-\frac d\alpha}
\end{align*}
Thus for $\Omega \geq 1$, we have
\[G^{(\beta)}(t, x, y) \asymp C_{T, d, \beta, \alpha} t^{-\frac{d\beta}\alpha} \Omega^{-1-\frac d\alpha},\]
as claimed.\end{proof}

Next we look at the spatial derivatives, where we consider separately small and large (but finite)~time.

\begin{prop}
Under the same assumptions as  Theorem \ref{thm:fracevollocal} and  Theorem \ref{thm:localderivativesstablelike}, the spatial derivatives of the Green's function $G^{(\beta)}_{\psi_\alpha, x}(t, x, y)$ for $(t,x,y)\in (0, 1)\times \R^d \times \R^d$ satisfy,
\begin{itemize}
\item For $\Omega \leq 1$,
\be\label{eqthm:stbllikederiv1}
\left|\pderiv{^k}{x_{i_1}\cdots\partial x_{i_k}}G_{\psi_\alpha, x}^{(\beta)}(t, x, y)\right|\leq ~C
\left\{
\begin{array}{lc}
t^{-\frac{(d+k)\beta}\alpha}, & d + k < \alpha,\\
t^{-\beta}(|\log \Omega | + 1),& d + k = \alpha,\\
t^{-\frac{(d+k)\beta}\alpha} \Omega^{1-\frac{d+k}\alpha},& d + k > \alpha.
\end{array}
\right.
\ee
for all $k\leq l$ and all indicies $i_i, \cdots, i_k$.
\item For $1 \leq \Omega \leq t^{-\beta}$,
\be\label{eqthm:stbllikederiv2}
\left|\pderiv{^k}{x_{i_1}\cdots\partial x_{i_k}}G_{\psi_\alpha, x}^{(\beta)}(t, x, y)\right|\leq
C t^{-\frac{(d+k)\beta}\alpha}\Omega^{-1-\frac{(d+k)}\alpha}
\ee
for all $k\leq l$ and all indicies $i_i, \cdots, i_k$.
\item For $\Omega \geq t^{-\beta}$,
\be\label{eqthm:stbllikederiv3}
\left|\pderiv{^k}{x_{i_1}\cdots\partial x_{i_k}}G_{\psi_\alpha, x}^{(\beta)}(t, x, y)\right|\leq
C t^{-\frac{d\beta}\alpha}\Omega^{-1-\frac{d}\alpha}
\ee
for all $k\leq l$ and all indicies $i_i, \cdots, i_k$.
\end{itemize}
\end{prop}
\begin{proof}
Splitting up the stable density followed by using the estimates (\ref{eq:localderivativestablelike}) and (\ref{eq:globalderivativestablelike}) we have
\begin{align*}
\left|\pderiv{^k}{x_{i_1}\cdots\partial x_{i_k}}G_{\psi_\alpha, x}^{(\beta)}(t, x, y)\right| \leq &~ c_\beta \int_0^1 \left|\pderiv{^k}{x_{i_1}\cdots\partial x_{i_k}}G_{\psi_\alpha, x}(t^\beta z, x, y)\right|~\mathrm dz\\
&+ c_\beta \int_1^\infty \left|\pderiv{^k}{x_{i_1}\cdots\partial x_{i_k}}G_{\psi_\alpha, x}^{(\beta)}(t^\beta z, x, y)\right| z^{-1 -\frac 1\beta} f_\beta(z^{-\frac 1\beta})~\mathrm dz\\
=:&~I_1 + I_2,
\end{align*}
where
\[I_1 := C \int_0^1 \min\left(t^{-\frac{(d+k)\beta}\alpha}\Omega^{-1-\frac{d+k}\alpha}z, t^{-\frac{(d+k)\beta}\alpha} z^{-\frac{d+k}\alpha}\right)~\mathrm dz,\]
and
\[I_2 := C \int_1^\infty \max((t^\beta z)^{-\frac k\alpha}, 1)\min\left(t^{-\frac{d\beta}\alpha}\Omega^{-1-\frac{d}\alpha}z, t^{-\frac{d\beta}\alpha} z^{-\frac{d}\alpha}\right)z^{-1-\frac 1\beta}e^{ct^\beta z}f_\beta(z^{-\frac 1\beta})~\mathrm dz.\]

Now note that since $t \in (0, 1)$, the integral in $I_1$ is the same as that one appearing in (\ref{eqprfdef:stablederivi1}), and thus for $\Omega \leq 1$,
\begin{align*}I_1 = &~ C t^{-\frac{(d+k)\beta}\alpha}\Omega^{-1-\frac{d+k}\alpha}\int_0^\Omega z~\mathrm dz+Ct^{-\frac{(d+k)\beta}\alpha}\int_\Omega^1 z^{-\frac{d+k}\alpha}~\mathrm dz\\
\leq &~C
\left\{
\begin{array}{lc}
t^{-\frac{(d+k)\beta}\alpha}, & d + k < \alpha,\\
t^{-\beta}(|\log \Omega | + 1),& d + k = \alpha,\\
t^{-\frac{(d+k)\beta}\alpha} \Omega^{1-\frac{d+k}\alpha},& d + k > \alpha.
\end{array}
\right.
\end{align*}

For $\Omega \geq 1$,
\be\label{eqprfdef:stbllikederivsmall}
I_1 = Ct^{-\frac{(d+k)\beta}\alpha}\Omega^{-1-\frac{d+k}\alpha}\int_0^1 z~\mathrm dz = \frac C2 t^{-\frac{(d+k)\beta}\alpha}\Omega^{-1-\frac{(d+k)\beta}\alpha}.
\ee

Turning to $I_2$, we need to consider some different cases.

\underline{\bf Case 1: $\Omega \leq 1$.}
In this case
\begin{align*}
I_2 = &~ C t^{-\frac{(d+k)\beta}\alpha}\int_1^{t^{-\beta}} z^{-\frac{d+k}\alpha - 1 +\frac 1{2(1-\beta)}}\exp\{ct^\beta z -c_\beta z^{\frac 1{1-\beta}}\}~\mathrm dz\\
&~+C t^{-\frac{d\beta}\alpha}\int_{t^{-\beta}}^\infty z^{-\frac d\alpha -1 + \frac 1{2(1-\beta)}}\exp\{ct^\beta z - c_\beta z^{\frac 1{1-\beta}}\}~\mathrm dz\\ 
\leq&~C t^{-\frac{(d+k)\beta}\alpha} \int_1^\infty z^{-\frac{d + k}\alpha - 1 + \frac 1{2(1-\beta)}}\exp\{-C_\beta z^{\frac 1{1-\beta}}\}~\mathrm dz\\
 &~+ C t^{\beta - \frac \beta{2(1-\beta)}}\exp\{-Ct^{-\frac{\beta}{1-\beta}}\}\\
\leq &~ C_{\beta, d, \alpha, k} t^{-\frac{(d+k)\beta}\alpha}.
\end{align*}

Combining this with the estimate for $I_1$ shows (\ref{eqthm:stbllikederiv1}).

\underline{\bf Case 2: $1 \leq \Omega \leq t^{-\beta}$.}
In this case we have
\begin{align*}
I_2 =&~ C t^{-\frac{(d+k)\beta}\alpha}\Omega^{-1-\frac{d+k}\alpha}\int_1^{\Omega}z^{\frac 1{2(1-\beta)}}\exp\{ct^\beta z - c_\beta z^{\frac 1{1-\beta}}\}~\mathrm dz\\
&~+C t^{-\frac{(d+k)\beta}\alpha}\int_{\Omega}^{t^{-\beta}} z^{-\frac{d+k}\alpha -1+\frac 1{2(1-\beta)}}\exp\{ct^\beta z - c_\beta z^{\frac 1{1-\beta}}\}~\mathrm dz\\
&~+C t^{-\frac{d\beta}\alpha}\int_{t^{-\beta}}^\infty z^{-\frac d\alpha - 1 + \frac 1{2(1-\beta)}}\exp\{ct^\beta z - c_\beta z^{\frac 1{1-\beta}}\}~\mathrm dz\\
\leq &~C t^{-\frac{(d+k)\beta}\alpha}\Omega^{-1-\frac{d+k}\alpha}\int_1^\infty z^{\frac 1{2(1-\beta)}}\exp\{ct^\beta z - c_\beta z^{\frac 1{1-\beta}}\}~\mathrm dz\\
&~+ C t^{-\frac{(d+k)\beta}\alpha}\int_\Omega^\infty z^{-\frac{d+k}\alpha -1 + \frac 1{2(1-\beta)}}\exp\{ct^\beta z - c_\beta z^{\frac 1{1-\beta}}\}~\mathrm dz\\
&~+ C t^{\beta - \frac 1{2(1-\beta)}}\exp\{-C t^{-\frac{\beta}{1-\beta}}\}\\
\leq &~C_\beta t^{-\frac{(d+k)\beta}\alpha}\Omega^{-1-\frac{d+k}\alpha}\\
&~+ C_{d, \alpha, \beta, l} t^{-\frac{(d+k)\beta}\alpha} \Omega^{-\frac{d+k}\alpha - \frac 1{2(1-\beta)}}\exp\{-C \Omega^{\frac 1{1-\beta}}\}\\
&~+ C t^{\beta - \frac 1{2(1-\beta)}}\exp\{-C t^{-\frac{\beta}{1-\beta}}\}\\
\leq &~C_{\beta, d, \alpha, k} t^{-\frac{(d+k)\beta}\alpha} \Omega^{-1-\frac{d+k}\alpha}.
\end{align*}

Combining this with (\ref{eqprfdef:stbllikederivsmall}) shows (\ref{eqthm:stbllikederiv2}).

\underline{\bf Case 3: $t^{-\beta} \leq \Omega$.}
\begin{align*}
I_2 =&~C t^{-\frac{(d+k)\beta}\alpha}\Omega^{-1-\frac{d+k}\alpha}\int_1^{t^{-\beta}}z^{\frac 1{2(1-\beta)}}\exp\{ct^\beta z - c_\beta z^{\frac 1{1-\beta}}\}~\mathrm dz\\
& +~C t^{-\frac{d\beta}\alpha}\Omega^{-1-\frac d\alpha}\int_{t^{-\beta}}^\Omega z^{\frac 1{2(1-\beta)}}\exp\{ ct^\beta z - c_\beta z^{\frac 1{1-\beta}}\}~\mathrm dz\\
& +~ C t^{-\frac{d\beta}\alpha} \int_\Omega^\infty z^{-\frac d\alpha - 1 +\frac 1{2(1-\beta)}}\exp\{ct^\beta z- c_\beta z^{\frac 1{1-\beta}}\}~\mathrm dz\\
\leq &~ C_{\beta} t^{-\frac{(d+k)\beta}\alpha}\Omega^{-1-\frac{d+k}\alpha}\\
&+~C_\beta t^{-\frac{d\beta}\alpha}\Omega^{-1-\frac d\alpha}\\
&+~C_{d, \beta, \alpha, l} t^{-\frac{d\beta}\alpha} \Omega^{-\frac d\alpha - \frac 1{2(1-\beta)}}\exp\{-C\Omega^{\frac 1{1-\beta}}\}\\
\leq &~ C_{d, \beta, \alpha, k} t^{-\frac{d\beta}\alpha}\Omega^{-1-\frac d\alpha}.
\end{align*}

Finally combining this with (\ref{eqprfdef:stbllikederivsmall}) shows (\ref{eqthm:stbllikederiv3}).
\end{proof}
Next, for large (finite) time.
\begin{prop}
Under the same assumptions as Theorem \ref{thm:fracevollocal} and  Theorem \ref{thm:localderivativesstablelike}, then for fixed $T> 0$, the following estimates hold for the spatial derivatives of the Green's function $G^{(\beta)}_{\psi_\alpha, x}(t, x, y)$ for $(t,x,y)\in (1, T)\times \R^d \times \R^d$,
\begin{itemize}
\item For $\Omega\leq 1$,
\be\label{eqthm:sbllikederivlrge1}
\left|\pderiv{^k}{x_{i_1}\cdots\partial x_{i_k}}G_{\psi_\alpha, x}^{(\beta)}(t, x, y)\right|\leq ~C_{T, d, \beta, \alpha, k}
\left\{
\begin{array}{lc}
1, & d + k < \alpha,\\
t^{-\beta}(|\log \Omega | + 1),& d + k = \alpha,\\
|x-y|^{\alpha - d- k},&d + k > \alpha,
\end{array}
\right.\\
\ee
for all $k\leq l$ and all indicies $i_i, \cdots, i_k$.

\item For $\Omega \geq 1$,
\be\label{eqthm:sbllikederivlrge3}
\left|\pderiv{^k}{x_{i_1}\cdots\partial x_{i_k}}G_{\psi_\alpha, x}^{(\beta)}(t, x, y)\right|\leq C |x-y|^{-\alpha - d},
\ee
for all $k\leq l$ and all indicies $i_i, \cdots, i_k$.
\end{itemize}
\end{prop}
\begin{proof}
As usual we first use \ref{eq:stabledensityineq},
\begin{align*}
\left|\pderiv{^k}{x_{i_1}\cdots\partial x_{i_k}}G_{\psi_\alpha, x}^{(\beta)}(t, x, y)\right| \leq & c_\beta \int_0^1 \left|\pderiv{^k}{x_{i_1}\cdots\partial x_{i_k}}G_{\psi_\alpha, x}(t^\beta z, x, y)\right|~\mathrm dz\\
&+ c_\beta \int_1^\infty \left|\pderiv{^k}{x_{i_1}\cdots\partial x_{i_k}}G_{\psi_\alpha, x}(t^\beta z, x, y)\right| z^{-1 -\frac 1\beta} f_\beta(z^{-\frac 1\beta})~\mathrm dz
\end{align*}

Next, we use the estimate (\ref{eq:localderivativestablelike}) for the first term and (\ref{eq:globalderivativestablelike}) for the second. Note that since \mbox{$t \in (1, T)$,~then}
\begin{align*}
\left|\pderiv{^k}{x_{i_1}\cdots\partial x_{i_k}}G_{\psi_\alpha, x}^{(\beta)}(t, x, y)\right| \leq & ~c \int_0^1 \min\left((t^\beta z)^{-\frac{d+k}\alpha},t^{-\frac{(d+k)\beta}\alpha}\Omega^{-1-\frac{d+k}\alpha}z \right)~\mathrm dz\\
& + c\int_1^\infty \min\left((t^\beta z)^{-\frac d\alpha}, t^{-\frac {d\beta}\alpha}\Omega^{-1-\frac d\alpha}z\right) z^{-1 -\frac 1\beta}e^{ct^\beta z}f_\beta(z^{-\frac 1\beta})~\mathrm dz\\
:=&~ I_1 + I_2
\end{align*}

The integral in $I_1$ is the same as that one appearing in (\ref{eqprfdef:stablederivi1}), and thus for $\Omega \leq 1$,
\begin{align*}I_1 = &~ C t^{-\frac{(d+1)\beta}\alpha}\Omega^{-1-\frac{d+1}\alpha}\int_0^\Omega z~\mathrm dz+Ct^{-\frac{(d+1)\beta}\alpha}\int_\Omega^1 z^{-\frac{d+1}\alpha}~\mathrm dz\\
\leq &~C
\left\{
\begin{array}{lc}
t^{-\frac{(d+k)\beta}\alpha}, & d + k < \alpha,\\
t^{-\beta}(|\log \Omega | + 1),& d + k = \alpha,\\
t^{-\frac{(d+k)\beta}\alpha} \Omega^{1-\frac{d+k}\alpha},& d + k > \alpha.
\end{array}
\right.
\end{align*}

However in this situation $t\in (1, T)$, so $t$ is away from both $0$ and $\infty$. Thus, recalling that $\Omega = |x-y|^\alpha t^{-\beta}$,
\[I_1 \leq C_{T,d, \beta, k,\alpha}
\left\{
\begin{array}{lc}
1, & d + k < \alpha,\\
t^{-\beta}(|\log \Omega | + 1),& d + k = \alpha,\\
|x-y|^{\alpha-d-k},& d + k > \alpha.
\end{array}
\right.
\]

For $\Omega \geq 1$,
\[I_1 = Ct^{-\frac{(d+1)\beta}\alpha}\Omega^{-1-\frac{d+k}\alpha}\int_0^1 z~\mathrm dz = \frac C2 t^{-\frac{(d+k)\beta}\alpha}\Omega^{-1-\frac{d+k}\alpha}\leq C_{T, \beta, d, \alpha, k} |x-y|^{-\alpha - d- k}.\]

Furthermore, the integral $I_2$ is the same as the one defined as $I^{up}_2$ in (\ref{eqprfdef:stablelikeiup}), thus for $\Omega \leq 1$,
\[I_2 \leq C t^{-\frac{d\beta}\alpha} \leq C_{d, \beta, \alpha, T} \]

So for $\Omega \leq 1$,
\[\left|\pderiv{^k}{x_{i_1}\cdots\partial x_{i_k}}G_{\psi_\alpha, x}^{(\beta)}(t, x, y)\right| \leq I_1 + I_2 \leq C_{T, d, \beta, \alpha, k} \left\{
\begin{array}{lc}
1, & d + k < \alpha,\\
t^{-\beta}(|\log \Omega | + 1),& d + k = \alpha,\\
|x-y|^{\alpha - d - k}, & d + k > \alpha,
\end{array}
\right.
\]
which thus gives (\ref{eqthm:sbllikederivlrge1}). Finally for $\Omega \geq 1$, using (\ref{eqprf:stbllikei2uplrge}),
\[I_2 \leq Ct^{-\frac{d\beta}\alpha}\Omega^{-1-\frac d\alpha}\leq C_{T, d, \beta, \alpha} |x-y|^{-\alpha - d}.\]
thus combining the estimates for $I_1$ and $I_2$ for $\Omega \geq 1$ gives us (\ref{eqthm:sbllikederivlrge3}).
\end{proof}

\section{Generalised Evolution Equations}
In this last section, we look at the following generalised evolution,
\be\label{eq:mixedevolution}
\left\{
\begin{array}{cr}
-D^{(\nu)}_{0+*} u(t,x) = Au(t,x),&(0, \infty)\times \R^d\\
u(0, x)  = \phi(x),&\{0\}\times \R^d,
\end{array}
\right.
\ee
where $D^{(\nu)}_{0+*}$ is the Caputo-type operator
\[D^{(\nu)}_{0+*} u(t) = -\int_0^{t}(f(t-r) - f(t))\nu(t, \mathrm dr) -(f(0)-f(t))\int_t^\infty \nu(t, \mathrm dr).\]

Here $\nu(t, \cdot)$ is a L\'evy transition kernel that satisfies $\sup_t\int \min(1, r)\nu(t, \mathrm dr) < \infty$.
The solution to Equation (\ref{eq:mixedevolution}) is given by
\[u(t,x) = E_{(\nu), t}(A)\phi(x),\]
where $E_{(\nu), t}(A)$ is the \emph{operator-valued generalised Mittag-Leffler} function which is defined by the operator-valued integral
\be\label{eq:generalisedopvalML}
E_{(\nu),t}(A) = \int_0^\infty e^{As} \mathrm d_s\left(\int_{-\infty}^t G_{(\nu)}(s, t, \mathrm dr)\right) = 1 + A \Pi^{-A}_{(\nu)}(t, [0, t]),
\ee
where $\Pi^{-A}_{(\nu)}$ is the operator-valued potential measure of the semigroup $T^{(\nu)}_t e^{tA}$ generated by \mbox{$(-D^{(\nu)}_{0+*} - A)$},
\[\Pi^{-A}_{(\nu)}(t,\mathrm dr) = \int_0^\infty e^{As} \mathrm ds~G_{(\nu)}(s, t, \mathrm dr).\]

Then we can rewrite this solution to get the Green's function,
\begin{align*}
E_{(\nu), t}(A)\phi(x) &= \int_{\R^d}\phi(y)\int_0^\infty G^{A}(s,x,y)\pderiv{}{s}\left(\int_t^\infty G_{(\nu)}(s,t,\mathrm dr)\right)\mathrm ds\mathrm dy\\
& = \int_{\R^d}\phi(y)G^{(\nu)}_A(t, x, y)~\mathrm dy,
\end{align*}
where $G^{(\nu)}_A$ is the Green's function of the evolution Equation (\ref{eq:mixedevolution}) given by
\[G^{(\nu)}_{A}(t, x, y) := \int_0^\infty G^A(s,x,y) \pderiv{}{s}\left(\int_t^\infty  G^{(\nu)}(s,t,\mathrm dr)\right)\mathrm ds.\]

We will use the following comparison principle from \cite{kolokoltsov2019mixed}.
\begin{thm}
Let $\nu$ and $\tilde{\nu}$ be two L\'evy measures satisfying
\[\nu(t, \mathrm dr) \geq \tilde{\nu}(\mathrm dr),\]
\[\sup_t \int_0^\infty \min(1, r)\nu(t,\mathrm dr) < \infty, \quad \int_0^\infty \min(1, r)\tilde{\nu}(\mathrm dr) < \infty,\]
and $\nu(t, (0, \infty)) = \tilde{\nu}((0, \infty)) = \infty$. Then for any non-increasing function $f$ we have the comparison principle for the semigroups:
\[T_t^{\nu} f \geq T^{\tilde{\nu}}_t f,\]
where $T^{(\nu)}_t$ and $T^{\tilde{\nu}}_t$ are the semigroups generated by
\[D^{(\nu)} f(t) =-\int_0^\infty (f(t - r) - f(t))\nu(t,\mathrm dr)\]
and
\[D^{(\tilde{\nu})} f(t) =-\int_0^\infty (f(t - r) - f(t))\tilde{\nu}(\mathrm dr)\]
respectively. Moreover, the potential measures of the semigroups $T_t^\nu$ and $T_t^{\tilde{\nu}}$ satisfy the comparison principle,
\[U^{(\nu)}(t,[0, t]) \leq U^{(\tilde{\nu})}([0, t]).\]
\end{thm}
A direct application of this comparison principle gives us the main result of this article.

\begin{thm}
Let $A$ be one of the spatial operators from (\ref{eq:fracdiffusion}), Theorem \ref{thm:stableglobal}, (\ref{eq:Lgeneraloperator}) or (\ref{eq:fracevolvar}) along with their relevant assumptions. Let $\nu(t,\mathrm ds)$ be a L\'evy transition kernel which has upper and lower bounds of $\beta$-fractional~type,
\[(-1/\Gamma(-\beta_1)) C_\nu s^{-1-\beta_1}\mathrm ds\leq\nu(t, \mathrm ds) \leq (-1/\Gamma(-\beta_2)) C_\nu s^{-1-\beta_2}\mathrm ds,\]
for some $\beta_1, \beta_2\in(0, 1)$ and $C_\nu > 0$. Then
\[c_2 E_{\beta_2}(At^{\beta_2})\phi(x) \leq E_{(\nu), t}(A)\phi(x) \leq c_1 E_{\beta_1}(At^{\beta_1})\phi(x),\]
for a non-increasing function $\phi$, where
\[E_{(\nu), t}(A) \phi(x) = \int_{\R^d}\phi(y) G^{(\nu)}_A(t,x,y)~\mathrm dy,\]
and
\[E_{\beta}(At^\beta) \phi(x) = \int_{\R^d}\phi(y) G^{(\beta)}_A(t,x,y)~\mathrm dy.\]
\end{thm}
\begin{proof}
This follows from the formula (\ref{eq:generalisedopvalML}) and an application of the comparison principle for potential operators.
\end{proof}
Thus the estimates obtained in Theorem \ref{thm:globaldivergenceestimate}, Theorem \ref{thm:stableglobal}, Theorem \ref{thm:fracdiffusionlocal} and Theorem \ref{thm:fracevollocal} can be used to estimate solutions of generalised evolutions (\ref{eq:mixedevolution}).
\begin{rem}
In order to see why this result is expected, let us give some intuition behind the comparison principle. The assumption that the L\'evy kernel is bounded below by the L\'evy kernel of a $\beta$-stable subordinator, means that the L\'evy subordinator generated by the operator
\[D^{(\nu)}_+ f(t) = -\int_0^\infty(f(x-y)-f(x))\nu(\mathrm dy),\]
has on average jumps that are larger than those of the process generated by
\[D^{\beta}_+ f(t) = \int_0^\infty (f(x-y)-f(x))[y^{1+\beta}\Gamma(-\beta)]^{-1}~\mathrm dy.\]

So on the sample paths level, the jumps of $X^{(\nu)}$ will typically be larger than those of $X^{\beta}$, which means that the inverse process of $X^{(\nu)}$ will typically be constant for longer times than the inverse process of $X^{\beta}$. Thus when we subordinate the spatial process, $Y(t)$, generated by the operator $A$ by the inverse subordinator given by
\[S_t^\nu := \inf\{s \geq 0: X_{s}^\nu \geq t\},\]
and compare its paths to the spatial process subordinated by an inverse stable subordinator $S^\beta_t$, we will see that $Y(S_t^\nu)$ is dominated by $Y(S_t^\beta)$ in the sense that $Y(S_t^\nu)$ will have longer trapping times.
\end{rem}

\section{Conclusion}
In this article, we have looked at two-sided estimates for the Green's function of fractional evolution equations of the form
\[D^{\beta}u(t, x) = Lu(t, x), \quad u(0, x) = Y(x).\]

The solution of such fractional evolution equations can be written with the help of operator-valued Mittag-Leffler functions,
\begin{align*}
u(t, x) = E_\beta (t^\beta L)Y(x) &= \int_0^\infty e^{zt^\beta L}Y(x) z^{-1-\frac 1\beta} w_\beta(z^{-\frac 1\beta})~\mathrm dz\\
&= \int_{\R^d} Y(y) \int_0^\infty G_L(t^\beta z, x, y)z^{-1-\frac 1\beta} w_\beta (z^{-\frac 1\beta})~\mathrm dz\\
& = \int_{\R^d} Y(y) G^{(\beta)}_L(t, x, y)~\mathrm dz.
\end{align*}

We have given two-sided estimates for the Green's function $G^{(\beta)}_L(t, x, y)$ (and its spatial derivatives) in several different situations. The situations can be split up into two broad cases: when the Green's function $G_L$ associated with $L$ does or does not have known global in time estimates. In those two cases, we consider generators of diffusion processes in Theorems \ref{thm:globaldivergenceestimate} and \ref{thm:fracdiffusionlocal}; and we consider generators of stable and stable-like processes in Theorems \ref{thm:stableglobal} and \ref{thm:fracevollocal}. Finally, we looked at generalised evolution equations where the operator acting on the time variable is given by a Caputo-type operator
\[D^{(\nu, t)}_0u(t) = \int_0^t (u(t - s) - u(t))\nu(t, \mathrm ds) + \int_t^\infty (u(0) - u(t))\nu(t, \mathrm ds).\]

We concluded that solutions to generalised evolution equations of the form
\be\label{eqconc:general}D^{(\nu, t)}_0 u(t, x) = Lu(t, x), \quad u(0, x) = Y(x),
\ee
where $\nu(t, \mathrm ds)$ is a L\'evy-type kernel which for fixed $t$ is comparable to the L\'evy measure of a $\beta$-stable subordinator, could be estimated using the estimates obtained for $G^{(\beta)}_L$. Then whenever one is looking at evolution equations of the form (\ref{eqconc:general}), or, from the probabilistic point of view, at stochastic processes generated by $-D^{(\nu)} - L$, then under the assumption that $\nu$ is comparable to $\beta$-stable, the estimates shown in this article can be used to gain a lot of information. 

Note that in this article we have viewed $G^{(\beta)}_L(t,x,y)$ as the Green's function of the evolution~equation
\[D^\beta u(t,x) = Lu(t,x).\]

Probabilistically speaking, $G_L^{(\beta)}$ are the transition densities of the process $X_t^{L, \beta}$ generated by $-D^\beta - L$. The process $X_t^{L, \beta}$ is the subordination of the process generated by $L$ by the inverse of the process generated by $D^\beta$. In this view one could use the estimates in this article to obtain sample path properties of a subordinated process $X^{L, \beta}_t$.

\section{Asymptotic Methods}\label{sec:asymptoticmethods}
 The main idea of the Laplace method for estimating integrals of the form 
\[\int g(x)\exp\{-\lambda h(x)\}~\mathrm dx,\]
is that the major contribution to the asymptotic behaviour comes from a neighbourhood around the point at which the function $h(x)$ in the exponent attains its minimum value. Outside this neighbourhood the contribution is exponentially small, and so when one proves asymptotic formulas using Laplace methods, the integrals are split up into the neighbourhood around which the major contribution occurs (or around each such neighbourhood, if $-h(x)$ is not unimodal) and the regions for which the approximation error is exponentially small. Although we focus on integrals over some interval $(a, \infty)$, the point is that extending the interval only introduces exponentially small errors and so the value of the integral over a larger interval essentially the same. Our standard references for asymptotic methods are \cite{de1970asymptotic}, \cite{murray2012asymptotic} or \cite{fedorjuk1987asymptotics}.
Consider the integral
\[
\int_b^\infty \exp\{-\lambda h(x)\}~\mathrm dx, \quad b > 0.
\]

Let us assume that $h$ is a real continuous function, and that it attains its minimum at the boundary point $b$, that $h'(b)$ exists and $h'(b) > 0$. Moreover assume that $h(x) > h(b)$ (for $x > b$) and $h(x) \rightarrow \infty$ as $x\rightarrow \infty$. We will not recount the proof, but we state the asymptotic formula, 
\be\label{eq:laplaceboundary}\int_b^\infty g(x)\exp\{-\lambda h(x)\}~\mathrm dx \sim g(b)(\lambda h'(b))^{-1}\exp\{-\lambda h(b)\}, \quad \lambda\rightarrow \infty.\ee

On the other hand, if the function $h$ has a minimum on the interior of the interval $(b, \infty)$, say at the point $\tilde{b}\in(b, \infty)$. Finally, assume that the derivative $h'(x)$ exists in some neighbourhood of $x = \tilde{b}$, that $h''(\tilde{b})$ exists and that $h''(\tilde{b}) > 0$. Then
\be\label{eq:laplaceinterior}
\int_b^\infty g(x)\exp\{-\lambda h(x)\}~\mathrm dx\sim g(\tilde{b})\sqrt{\frac{2\pi}{\lambda h''(\tilde{b})}}\exp\{-\lambda h(\tilde{b})\}, \quad \lambda\rightarrow \infty.
\ee

\begin{prop}\label{prop:asymptoticcomputation}
Let $a> 0$, $N\in \R$, $c > 0$ and $\Omega \geq 1$. Then the following asymptotic formula holds as $\Omega \rightarrow \infty$,
\[\int_0^1 w^N \exp\{-\Omega w - c w^{-a}\}~\mathrm dw\sim C_1(a, N, c) \Omega^{-\frac{2(N+1) + a}{2(a+1)}}\exp\left\{-C_2(c, a)\Omega^{\frac{a}{a+1}}\right\},\]
where $C_1(a, N, c) = (ac)^{\frac{2(N+1) - 1}{2(a+1)}}\sqrt{\frac{2\pi}{a+1}}$, and $C_2(c, a) = (ac)^{\frac 1{a+1}}[1+ a^{-1}]$.
\end{prop}
\begin{proof}
Define
\[J(\Omega) := \int_0^1 w^{N} \exp\{-w\Omega - cw^{-a}\}~\mathrm dw,\]
and let $h(w) = -w\Omega - cw^{-a}$. Differentiating $h$ with respect to $w$, one finds the maximum of $h$ at
\[w = w_* := \left(\frac{\Omega}{ac}\right)^{-\frac 1{a+1}}.\]

Now the trick is to make the substitution $w = w_* s$ in the integral $J(\Omega)$, to obtain
\begin{align*}
J(\Omega) &= w_*^{N+1}\int_0^{w_*^{-1}} s^N \exp\{-w_* s\Omega - c(w_*s)^{-a}\}~\mathrm ds\\
& = w_*^{N+1}\int_0^{w_*^{-1}} s^N \exp\{-(\Omega^a ac)^{\frac 1{a+1}}[s + a^{-1}s^{-a}]\}~\mathrm ds.
\end{align*}

Now we are in a position to apply the asymptotic formula (\ref{eq:laplaceinterior}), with $g(s) = s^N$, $h(s) =  s + a^{-1}s^{-a}$ and $\lambda = (\Omega^a ac)^{\frac 1{a+1}}$. For this we need some derivatives of $h$,
\[h'(s) = 1- s^{-a-1},\]
\[h''(s) = (a + 1)s^{-a-2},\]
thus $h$ has a minimum at $s=1 \in (0, w_*^{-1})$. Finally applying (\ref{eq:laplaceinterior}) we have
\begin{align*}
J(\Omega) &\sim \left(\frac{\Omega}{ac}\right)^{-\frac{N+1}{a+1}}\sqrt{\frac{2\pi}{(\Omega^a ac)^{\frac 1{a+1}}(a+1)}}\exp\{-(\Omega^a ac)^{\frac 1{a+1}}[1 + a^{-1}]\}\\
&= C_1(a, N, c) \Omega^{-\frac{2(N+1) + a}{2(a+1)}}\exp\left\{-C_2(c, a)\Omega^{\frac{a}{a+1}}\right\},
\end{align*}
as required.
\end{proof}
Note that when applying the above asymptotic formula, we will not care so much what the constants are, only what they depend on. 

\bibliographystyle{alpha}
\bibliography{}

\newcommand{\etalchar}[1]{$^{#1}$}
\begin{thebibliography}{CKKW18}

\bibitem[Aro67]{aronson1967bounds}
D.~G. Aronson.
\newblock Bounds for the fundamental solution of a parabolic equation.
\newblock {\em Bulletin of the American Mathematical society}, 73(6):890--897,
  1967.

\bibitem[B{\v C}11]{barlow2011convergence}
Martin~T. Barlow and Ji{\u r}{\'\i} {\v C}ern{\i y}.
\newblock Convergence to fractional kinetics for random walks associated with
  unbounded conductances.
\newblock {\em Probability theory and related fields}, 149(3-4):675--677, 2011.

\bibitem[C{\etalchar{+}}15]{cabezas2015sub}
Manuel Cabezas et~al.
\newblock Sub-gaussian bound for the one-dimensional bouchaud trap model.
\newblock {\em Brazilian Journal of Probability and Statistics},
  29(1):112--131, 2015.

\bibitem[CKKW18]{chen2018heat}
Z.-Q. Chen, P.~Kim, T.~Kumagai, and J.~Wang.
\newblock Heat kernel estimates for time fractional equations.
\newblock In {\em Forum Mathematicum}, volume~30, pages 1163--1192. De Gruyter,
  2018.

\bibitem[CLY17]{cheng2017asymptotic}
Xing Cheng, Zhiyuan Li, and Masahiro Yamamoto.
\newblock Asymptotic behavior of solutions to space-time fractional
  diffusion-reaction equations.
\newblock {\em Mathematical Methods in the Applied Sciences}, 40(4):1019--1031,
  2017.

\bibitem[dB81]{de1970asymptotic}
N.~G. de~Bruijn.
\newblock {\em Asymptotic methods in analysis}.
\newblock Dover Publications, Inc., New York, third edition, 1981.

\bibitem[Die10]{diethelm2010analysis}
Kai Diethelm.
\newblock {\em The analysis of fractional differential equations: An
  application-oriented exposition using differential operators of Caputo type}.
\newblock Springer, 2010.

\bibitem[DS18]{deng2018exact}
Chang-Song Deng and Ren{\'e}~L Schilling.
\newblock Exact asymptotic formulas for the heat kernels of space and
  time-fractional equations.
\newblock {\em arXiv preprint arXiv:1803.11435}, 2018.

\bibitem[EIK04]{eidelman2004analytic}
S.~D. Eidelman, S.~D. Ivasyshen, and A.~N. Kochubei.
\newblock {\em Analytic methods in the theory of differential and
  pseudo-differential equations of parabolic type}, volume 152.
\newblock Springer Science \& Business Media, 2004.

\bibitem[EK04]{eidelman2004cauchy}
S.D Eidelman and A.N Kochubei.
\newblock Cauchy problem for fractional diffusion equations.
\newblock {\em Journal of differential equations}, 199(2):211--255, 2004.

\bibitem[Eva10]{evans2010partial}
L.C. Evans.
\newblock {\em Partial Differential Equations}.
\newblock Graduate studies in mathematics. American Mathematical Society, 2010.

\bibitem[Fed87]{fedorjuk1987asymptotics}
M.~V. Fedoryuk.
\newblock {\em Asymptotics: integrals and series}.
\newblock Mathematical Reference Library. ``Nauka'', Moscow, 1987.

\bibitem[Fel08]{feller2008introduction}
W.~Feller.
\newblock {\em An introduction to probability theory and its applications},
  volume~2.
\newblock John Wiley \& Sons, 2008.

\bibitem[GK08]{grigor2008dichotomy}
Alexander Grigoryan and Takashi Kumagai.
\newblock On the dichotomy in the heat kernel two sided estimates.
\newblock In {\em Analysis on Graphs and its Applications (P. Exner et
  al.(eds.)), Proc. of Symposia in Pure Math}, volume~77, pages 199--210, 2008.

\bibitem[GMSR01]{gorenflo2001fractional}
Rudolf Gorenflo, Francesco Mainardi, Enrico Scalas, and Marco Raberto.
\newblock Fractional calculus and continuous-time finance iii: the diffusion
  limit.
\newblock In {\em Mathematical finance}, pages 171--180. Springer, 2001.

\bibitem[HIK{\etalchar{+}}18]{hairer2018fractional}
Martin Hairer, Gautam Iyer, Leonid Koralov, Alexei Novikov, Zsolt Pajor-Gyulai,
  et~al.
\newblock A fractional kinetic process describing the intermediate time
  behaviour of cellular flows.
\newblock {\em The Annals of Probability}, 46(2):897--955, 2018.

\bibitem[KKdS18]{kochubei2018random}
A.N Kochubei, Y~Kondratiev, and J.L da~Silva.
\newblock From random times to fractional kinetics.
\newblock {\em arXiv preprint arXiv:1811.10531}, 2018.

\bibitem[KKdS19]{kochubei2019random}
Anatoly~N Kochubei, Yuri Kondratiev, and Jos{\'e}~Lu{\'\i}s da~Silva.
\newblock Random time change and related evolution equations.
\newblock {\em arXiv preprint arXiv:1901.10015}, 2019.

\bibitem[KKM16]{kelbert2016weak}
Mark Kelbert, Valentin Konakov, and St{\'e}phane Menozzi.
\newblock Weak error for continuous time markov chains related to fractional in
  time p (i) des.
\newblock {\em Stochastic Processes and their Applications}, 126(4):1145--1183,
  2016.

\bibitem[KL16]{kim2016asymptotic}
Kyeong-Hun Kim and Sungbin Lim.
\newblock Asymptotic behaviors of fundamental solution and its derivatives to
  fractional diffusion-wave equations.
\newblock {\em J. Korean Math. Soc}, 53(4):929--967, 2016.

\bibitem[Kol00]{kolokoltsov2000symmetric}
Vassili~N Kolokoltsov.
\newblock Symmetric stable laws and stable-like jump-diffusions.
\newblock {\em Proceedings of the London Mathematical Society}, 80(3):725--768,
  2000.

\bibitem[Kol11]{kolokoltsov2011markov}
Vassili~N Kolokoltsov.
\newblock {\em Markov Processes, Semigroups, and Generators}, volume~38.
\newblock Walter de Gruyter, 2011.

\bibitem[Kol15]{kolokoltsov2015fully}
V.~N. Kolokoltsov.
\newblock On fully mixed and multidimensional extensions of the caputo and
  riemann-liouville derivatives, related markov processes and fractional
  differential equations.
\newblock {\em Fractional Calculus and Applied Analysis}, 18(4):1039--1073,
  2015.

\bibitem[Kol17]{kolokoltsov2017chronological}
Vassili Kolokoltsov.
\newblock Chronological operator-valued feynman-kac formulae for generalized
  fractional evolutions.
\newblock {\em arXiv preprint arXiv:1705.08157}, 2017.

\bibitem[Kol19a]{kolokoltsov2019differential}
Vassili~N. Kolokoltsov.
\newblock {\em Differential equations on measures and functional spaces}.
\newblock Birkh\''auser, 2019.

\bibitem[Kol19b]{kolokoltsov2019mixed}
V.N Kolokoltsov.
\newblock The probabilistic point of view on the generalized fractional pdes.
\newblock {\em Fractional Calculus and Applied Analysis}, 2019.

\bibitem[KST06]{kilbas2006theory}
A~Anatolii~Aleksandrovich Kilbas, Hari~Mohan Srivastava, and Juan~J Trujillo.
\newblock {\em Theory And Applications of Fractional Differential Equations},
  volume 204.
\newblock Elsevier Science Limited, 2006.

\bibitem[KSZ17]{kemppainen2017representation}
Jukka Kemppainen, Juhana Siljander, and Rico Zacher.
\newblock Representation of solutions and large-time behavior for fully
  nonlocal diffusion equations.
\newblock {\em Journal of Differential Equations}, 263(1):149--201, 2017.

\bibitem[KV14]{kolokoltsovwell2014}
V.~Kolokoltsov and M.~Veretennikova.
\newblock Well-posedness and regularity of the cauchy problem for nonlinear
  fractional in time and space equations.
\newblock {\em Fractional Differential Calculus}, 2014.

\bibitem[LMS13]{leonenko2013correlation}
Nikolai~N Leonenko, Mark~M Meerschaert, and Alla Sikorskii.
\newblock Correlation structure of fractional pearson diffusions.
\newblock {\em Computers \& Mathematics with Applications}, 66(5):737--745,
  2013.

\bibitem[Mai10]{mainardi2010fractional}
Francesco Mainardi.
\newblock {\em Fractional calculus and waves in linear viscoelasticity: an
  introduction to mathematical models}.
\newblock World Scientific, 2010.

\bibitem[MRGS00]{mainardi2000fractional}
Francesco Mainardi, Marco Raberto, Rudolf Gorenflo, and Enrico Scalas.
\newblock Fractional calculus and continuous-time finance ii: the waiting-time
  distribution.
\newblock {\em Physica A: Statistical Mechanics and its Applications},
  287(3-4):468--481, 2000.

\bibitem[MS12]{meerschaert2012stochastic}
Mark~M Meerschaert and Alla Sikorskii.
\newblock {\em Stochastic models for fractional calculus}, volume~43.
\newblock Walter de Gruyter, 2012.

\bibitem[Mur84]{murray2012asymptotic}
J.~D. Murray.
\newblock {\em Asymptotic analysis}, volume~48 of {\em Applied Mathematical
  Sciences}.
\newblock Springer-Verlag, New York, second edition, 1984.

\bibitem[P{\`E}84]{porper1984two}
F.O Porper and S.~D {\`E}idel'man.
\newblock Two-sided estimates of fundamental solutions of second-order
  parabolic equations, and some applications.
\newblock {\em Russian Mathematical Surveys}, 39(3):119--178, 1984.

\bibitem[Pod98]{podlubny1998fractional}
Igor Podlubny.
\newblock {\em Fractional differential equations: an introduction to fractional
  derivatives, fractional differential equations, to methods of their solution
  and some of their applications}, volume 198.
\newblock Academic press, 1998.

\bibitem[Pol48]{pollard1948completely}
H.~Pollard.
\newblock The completely monotonic character of the mittag-leffler function
  $e_\alpha(-x)$.
\newblock {\em Bulletin of the American Mathematical Society},
  54(12):1115--1116, 1948.

\bibitem[Ric14]{richard2014fractional}
Herrmann Richard.
\newblock {\em Fractional calculus: an introduction for physicists}.
\newblock World Scientific, 2014.

\bibitem[SGM00]{scalas2000fractional}
Enrico Scalas, Rudolf Gorenflo, and Francesco Mainardi.
\newblock Fractional calculus and continuous-time finance.
\newblock {\em Physica A: Statistical Mechanics and its Applications},
  284(1-4):376--384, 2000.

\bibitem[SKM{\etalchar{+}}93]{samko1993fractional}
Stefan~G Samko, Anatoly~A Kilbas, Oleg~I Marichev, et~al.
\newblock Fractional integrals and derivatives.
\newblock {\em Theory and Applications, Gordon and Breach, Yverdon}, 1993,
  1993.

\bibitem[SSV12]{schilling2012bernstein}
Ren{\'e}~L Schilling, Renming Song, and Zoran Vondracek.
\newblock {\em Bernstein functions: theory and applications}, volume~37.
\newblock Walter de Gruyter, 2012.

\bibitem[Str88]{Stroock1988}
Daniel~W. Stroock.
\newblock Diffusion semigroups corresponding to uniformly elliptic divergence
  form operators.
\newblock In {\em S{\'e}minaire de Probabilit{\'e}s XXII}, pages 316--347.
  Springer, 1988.

\bibitem[Tar11]{tarasov2011fractional}
Vasily~E Tarasov.
\newblock {\em Fractional dynamics: applications of fractional calculus to
  dynamics of particles, fields and media}.
\newblock Springer Science \& Business Media, 2011.

\bibitem[UZ11]{uchaikin2011chance}
V.~V. Uchaikin and V.~M Zolotarev.
\newblock {\em Chance and stability: stable distributions and their
  applications}.
\newblock Walter de Gruyter, 2011.

\bibitem[VDBF12]{van2012potential}
C~Van Den~Berg and G.~Forst.
\newblock {\em Potential theory on locally compact abelian groups}, volume~87.
\newblock Springer Science \& Business Media, 2012.

\bibitem[Zol57]{zolotarev1957mellin}
V.~M. Zolotarev.
\newblock Mellin-stieltjes transforms in probability theory.
\newblock {\em Theory of Probability \& Its Applications}, 2(4):433--460, 1957.

\bibitem[Zol61]{zolotarev1961analytic}
V~M Zolotarev.
\newblock On analytic properties of stable distribution laws.
\newblock {\em Selected Translations in Mathematical Statistics and
  Probability}, 1:202--211, 1961.

\bibitem[Zol86]{zolotarev1986one}
V.~M. Zolotarev.
\newblock {\em One-dimensional stable distributions}, volume~65.
\newblock American Mathematical Soc., 1986.

\end{thebibliography}
\end{document}